\documentclass[11pt]{amsart}
\usepackage{amsmath,amssymb,amsfonts}
\usepackage[mathscr]{eucal}

\usepackage{enumerate}
\newenvironment{enumroman}{\begin{enumerate}[\upshape (i)]}
	{\end{enumerate}}

\usepackage[margin=1.4in]{geometry}

\theoremstyle{plain}
\newtheorem{theorem}[equation]{Theorem}

\newtheorem{prop}[equation]{Proposition}
\newtheorem{lemma}[equation]{Lemma}

\theoremstyle{definition}
\newtheorem{definition}[equation]{Definition}
\newtheorem{example}[equation]{Example}
\newtheorem{remark}[equation]{Remark}
\newtheorem*{thank}{Acknowledgments}

\numberwithin{equation}{section}

\input xy
\xyoption{all}

\newcommand{\Cat}{\mathcal Cat}
\newcommand{\colim}{\operatorname{colim}}
\newcommand{\cosk}{\operatorname{cosk}}
\newcommand{\Cpt}{\operatorname{Cpt}}
\newcommand{\css}{\mathcal{CSS}}
\newcommand{\cu}{\underline{c}}
\newcommand{\Deltaop}{\Delta^{\op}}

\newcommand{\heq}{\operatorname{heq}}
\newcommand{\Ho}{\operatorname{Ho}}
\newcommand{\Hom}{\operatorname{Hom}}
\newcommand{\id}{\operatorname{id}}
\newcommand{\Map}{\operatorname{Map}}
\newcommand{\map}{\operatorname{map}}
\newcommand{\ob}{\operatorname{ob}}
\newcommand{\op}{\operatorname{op}}

\newcommand{\Se}{\operatorname{Se}}
\newcommand{\se}{\operatorname{se}}
\newcommand{\Secat}{\mathcal Se\Cat}
\newcommand{\sesp}{\mathcal Se\Sp}
\newcommand{\Sets}{\mathcal Sets}
\newcommand{\sk}{\operatorname{sk}}
\newcommand{\Sp}{\mathcal Sp}
\newcommand{\SSets}{\mathcal{SS}ets}
\newcommand{\Thetanop}{\Theta_n^{\op}}
\newcommand{\Thetansp}{\Theta_n\css}
\newcommand{\tr}{\operatorname{tr}}
\newcommand{\uHo}{\underline{\Ho}}

\newcommand{\uMap}{\underline{\Map}}
\newcommand{\vu}{\underline{v}}
\newcommand{\wu}{\underline{w}}

\begin{document}
	
	\title[$(\infty,n)$-categories]{Discreteness and completeness for $\Theta_n$-models of $(\infty,n)$-categories}
	
	\author[J.E. Bergner]{Julia E. Bergner}
	
	\address{Department of Mathematics, University of Virginia, Charlottesville, VA 22904}
	
	\email{jeb2md@virginia.edu}
	
	\date{\today}
	
	
	
	\thanks{The author was partially supported by NSF grants DMS-1659931 and DMS-1906281.  Work was also done on this paper at the Isaac Newton Institute for Mathematical Sciences, Cambridge, during the Homotopy Harnessing Higher Structures program in 2018, supported by EPSRC grant no EP/K032208/1; at Cornell University while supported by the AWM Michler Prize in 2018; at MSRI while the author was in residence during the Derived Algebraic Geometry program in 2019, supported by NSF grant DMS-14400140; at the Higher Categories and Categorification program in 2022, supported by MSRI via NSF grant DMS-1928930 and held in partnership with the the Universidad Nacional Aut\'onoma de M\'exico; and at the Higher Symmetries and Quantum Field Theory program at the Aspen Center for Physics, supported by National Science Foundation grant PHY-1607611.}
	
	\begin{abstract}
		We establish cartesian model structures for variants of $\Theta_n$-spaces in which we replace some or all of the completeness conditions by discreteness conditions. We prove that they are all equivalent to each other and to the $\Theta_n$-space model, and we give a criterion for which combinations of discreteness and completeness give non-overlapping models.  These models can be thought of as generalizations of Segal categories in the framework of $\Theta_n$-diagrams.  In the process, we give a characterization of the Dwyer-Kan equivalences in the $\Theta_n$-space model, generalizing the one given by Rezk for complete Segal spaces.
	\end{abstract}
	
	\maketitle
	
	\tableofcontents
	
	\section{Introduction}
	
	An $(\infty,n)$-category should be a higher categorical structure with objects and $i$-morphisms for all $i \geq 1$, satisfying weak associativity and unitality, and such that all $i$-morphisms are weakly invertible for $i>n$.  There have been a number of different approaches to modeling such a structure as a concretely-defined mathematical object, including the Segal $n$-categories of Hirschowitz and Simpson \cite{hs} and Pelissier \cite{pel}, the $n$-fold complete Segal spaces of Barwick and Lurie \cite{luriecob}, and the $\Theta_n$-spaces of Rezk \cite{rezktheta}.  Models that blend features of multiple of these models have been given by the author and Rezk \cite{inftyn1}, \cite{inftyn2} and by Moser, Rasekh, and Rovelli \cite{mrr}.  Other models include the $n$-relative categories of Barwick and Kan \cite{bkrel} and the $\Theta_n$-sets of Ara \cite{ara}.   More recently, models based on Verity's complicial sets \cite{verity} have been developed and compared by Ozornova and Rovelli \cite{or}, \cite{oradj} and Loubaton \cite{loubaton}, and Campion, Kapulkin, and Yaehara have made progress with cubical models \cite{cky}.  Further comparisons between models have been given by Haugseng \cite{haug}, by Doherty, Kapulkin, and Maehara \cite{dky}, and, using an axiomatic approach, by Barwick and Schommer-Pries \cite{bsp}.  In the special case when $n=2$, there are further results by the author with Ozornova and Rovelli \cite{infty2comp}, and by Gagna, Harpaz, and Lanari \cite{ghl}.  Results when $n=1$ are now well-established, and include the comparisons of Barwick and Kan \cite{bk}, the author \cite{thesis}, Dugger and Spivak \cite{ds}, Joyal \cite{joyalqcat}, Joyal and Tierney \cite{jt}, Lurie \cite{lurie}, and the axiomatic approach of To\"en \cite{toen}.
	
	Many of these models are given by some kind of diagram of simplicial sets; for example, $n$-fold complete Segal spaces and Segal $n$-categories are given by multisimplicial diagrams, and $\Theta_n$-spaces are given by functors out of the category $\Thetanop$.  Such diagrams are required to satisfy $n$ different Segal conditions, which essentially encode an up-to-homotopy composition for each of the $n$ levels of (not necessarily invertible) morphisms.  
	
	However, Segal diagrams without further assumptions do not quite model $(\infty, n)$-categories, which should behave like iterated enriched categories and in particular have a discrete space of objects and discrete spaces of $k$-morphisms for all $1 \leq k <n$.  Without such assumptions, we get structures more reflective of $n$-categories internal to spaces.    
	
	In general, there are two ways to impose extra structure on Segal diagrams to get models for $(\infty,n)$-categories.  The first is straightforward: simply ask that the desired spaces in the diagram be discrete.  This approach was the one taken by Hirschowitz and Simpson in their definition of Segal $n$-categories \cite{hs}.  This simplicity of definition, however, comes at a cost.  Being discrete is a rather unnatural condition from the perspective of homotopy theory, and as such causes many complications in setting up appropriate model structures.
	
	Rezk took a different approach in his development of complete Segal spaces as models for $(\infty,1)$-categories, and asked instead for a completeness condition, which essentially asks that the space of objects be weakly equivalent to the subspace of morphisms that behave suitably like homotopy equivalences.  Thus, the data of the whole object space is already encoded in the space of morphisms, and does not give substantially new information.  This approach was generalized to higher $(\infty,n)$-categories via the $n$-fold complete Segal space and $\Theta_n$-space models, in which $n$ different completeness conditions are assumed.  For example, the space of 1-morphisms is required to be weakly equivalent to the space of 2-morphisms that are homotopy equivalences in an appropriate sense, and similarly for higher morphisms.
	
	Thus, when we look at existing diagrammatic models for $(\infty,n)$-categories, taking multisimplicial models has an accompanying choice of whether to impose discreteness or completeness conditions, whereas thus far the $\Theta_n$-model has only been considered with completeness conditions.  A natural question is whether there is a corresponding model given by $\Theta_n$-diagrams with discreteness conditions, and answering it is the primary motivation for this paper.
	
	However, another question arises: do we need to make a single choice of either discreteness or completeness conditions everywhere, or can these conditions be ``mixed and matched"?  In \cite{inftyn1} we show that Segal category objects in $\Theta_{n-1}$-spaces give a model for $(\infty,n)$-categories; such objects have one discreteness condition and $(n-1)$-completeness conditions in the setting of $\Delta \times \Theta_{n-1}$-diagrams.  Here we address this question for $\Theta_n$-diagrams, and show that discreteness can be imposed ``from the bottom up": we get distinct models for $(\infty,n)$-categories by taking $\Theta_n$-diagrams of simplicial sets with discreteness imposed at level $k$, for some fixed $0 \leq k <n$, and completeness imposed for any $k < i < n$.  Essentially, if we ask for discreteness at a given level of morphism, the spaces of all lower levels of morphisms are forced to be discrete also.
	
	This answer in the $\Theta_n$-diagram setting naturally leads us back to ask the analogous question in the context of multisimplicial diagrams, as well as hybrid diagrams, indexed by categories of the form $\Delta^i \times \Theta_{n-i}$ for some $0<i<n$, of which Segal category objects in $\Theta_{n-1}$-spaces are an example.  In a sequel paper \cite{discrete2}, we address such models, and in particular give an explicit comparison between $n$-fold complete Segal spaces and Segal $n$-categories via these hybrids, a result that has been assumed but that does not seem to be in the literature.
	
	The idea that $\Theta_n$-diagrams with discreteness assumptions should be a viable model for $(\infty,n)$-categories developed at the early stages of the comparison of $\Theta_n$-spaces with categories enriched in $\Theta_{n-1}$-spaces, which the author established with Rezk.  However, complications with understanding Dwyer-Kan equivalences, and a suitable completion functor generalizing the one for complete Segal spaces in \cite{rezk} seemed prohibitively difficult.  Fortunately, results from that comparison program with Rezk have greatly facilitated the development of this model and its comparison with $\Theta_n$-spaces, as we show here.
	
	Our motivation for this work is primarily aesthetic, in that we believe that establishing all possible options for these kind of models results in a satisfying picture of the choices we have for this flavor of models for $(\infty,n)$-categories.  However, it also seems that the restriction of the models here to the case of $(\infty,n)$-categories with a single object provide a good way to think about monoidal $(\infty, n-1)$-categories, a project currently being undertaken by Valentina Zapata Castro.  It would also be worth investigating whether examples originally modeled by Segal $n$-categories might be more compactly described using $\Theta_n$-diagrams rather than multisimplicial ones.  
	
	After some general background in Section 2, we review the Segal category and complete Segal space models in Section 3.  In Section 4, we recall Rezk's $\Theta_n$-space model, and in Section 5, we recall the development of Segal and complete Segal objects in $\Theta_n$-spaces.  We devote Section 6 to understanding Dwyer-Kan equivalences in $\Theta_n$-spaces.  In Section 7 we give a general treatment of making objects in general diagrams discrete, which we then use in Section 8 to give a model structure on the category of $\Theta_n$-diagrams with certain objects discrete.  In Section 9 we give our comparison result between these models and Rezk's original $\Theta_n$-space model.  Finally in Section 10 we give the proof of a deferred technical result.
	
	\thank{The impetus to establish the models here arose while I was writing \cite{survey}, so I would like to thank Haynes Miller for the invitation to write that paper.  The influence of Charles Rezk's ideas can be found throughout, especially as this work builds so extensively on our previous collaboration.  I have had many conversations with many people about this aspects of this project, including Clark Barwick, Clemens Berger, Lyne Moser, Viktoriya Ozornova, Nima Rasekh, and Martina Rovelli, and more recently, my students Miika Tuominen and Valentina Zapata Castro.  Finally, I'd like to thank the referee and editor for helpful comments on the paper.}
	
	\section{Some background}
	
	Our goal is to put model structures on certain categories of functors in such a way that the objects that are both fibrant and cofibrant give models for $(\infty,n)$-categories.  Here, we give a brief review of some of the model category tools that we need.
	
	The models that we consider in this paper are all given by functors from some small category to the category of simplicial sets.  Recall that a \emph{simplicial set} is a functor $\Deltaop \rightarrow \Sets$, where $\Deltaop$ is the opposite of the category of finite ordered sets, and $\Sets$ denotes the category of sets.  We denote by $\SSets$ the category of simplicial sets.
	
	First, we recall the classical model structure on the category of simplicial sets, originally due to Quillen \cite{quillen}.  The weak equivalences are the maps whose geometric realizations are weak homotopy equivalences of topological spaces, the cofibrations are the monomorphisms, and the fibrations are the Kan fibrations.  Since this model structure is the only one that we consider on the category of simplicial sets, we simply use the notation $\SSets$ to refer to it.
	
	Given a small category $\mathcal C$, there are two canonical model structures on the category $\SSets^\mathcal C$ of functors $\mathcal C \rightarrow \SSets$, both of which have weak equivalences given by levelwise weak equivalences of simplicial sets.  In the \emph{injective} model structure, we take the cofibrations to be given levelwise, whereas in the \emph{projective} model structure, we take the fibrations to be given levelwise.  In this paper, we are primarily interested in the injective model structure.  Notice in particular that all objects are cofibrant in this model structure.  However, the disadvantage of the injective model structure is that we generally do not have explicit descriptions of sets of generating cofibrations and acyclic cofibrations, only arguments for their existence.  
	
	When $\mathcal C$ has the additional structure of a Reedy category, there is a third option: the Reedy model structure, which also has levelwise weak equivalences but fibrations and cofibrations defined in terms of matching and latching objects \cite{reedy}, \cite[15.3.4]{hirsch}.  In the particularly nice situation when $\mathcal C$ has the structure of an elegant Reedy category in the sense of \cite{elegant}, the Reedy model structure coincides with the injective structure.  Thus, we have both advantages: the cofibrancy of all objects from the injective model structure but also the explicit generating cofibrations and acyclic cofibrations for the Reedy model structure.  All of the indexing categories that we consider in this paper, specifically the categories $\Theta_n^{\op}$ for all $n \geq 0$, are elegant Reedy categories.
	
	In nice cases, models for $(\infty,n)$-categories can be obtained by localizing this Reedy model structure with respect to a set of maps in such way that the fibrant objects are the local objects with respect to these maps.  Let us recall this process briefly.
	
	Recall that in any model category, there is a notion of \emph{homotopy mapping space} $\Map^h(X,Y)$ between any two objects $X$ and $Y$.  When the model category in question is simplicial, then the homotopy mapping spaces can be obtained by taking the simplicial mapping spaces of cofibrant-fibrant replacements of $X$ and $Y$, but they can be defined more generally \cite[\S 17]{hirsch}.  
	
	Now, let $\mathcal M$ be a model category, and $S$ a set of maps in $\mathcal M$.  A fibrant object $Z$ of $\mathcal M$ is $S$-\emph{local} if, for any map $A \rightarrow B$ in $S$, the induced map
	\[ \Map^h(B,Z) \rightarrow \Map^h(A,Z) \]
	is a weak equivalence of simplicial sets.  An arbitrary map $X \rightarrow Y$ of $\mathcal M$ is an $S$-\emph{local equivalence} if, for any $S$-local object $Z$, the induced map
	\[ \Map^h(Y, Z) \rightarrow \Map^h(X,Z) \]
	is a weak equivalence of simplicial sets.  If $\mathcal M$ is a sufficiently nice model category, then there exists a model structure $\mathcal L_S \mathcal M$ on the underlying category of $\mathcal M$ in which the weak equivalences are the $S$-local equivalences, the cofibrations are the same as those of $\mathcal M$, and the fibrant objects are the $S$-local objects \cite[4.1.1]{hirsch}.
	
	However, in many examples we consider in this paper, we do not have a suitable model structure from which we can obtain our desired model structure as a localization.  We thus have to prove the existence of such a model structure from scratch.  We thus include the following recognition principle for cofibrantly generated model categories, originally due to Kan.
	
	\begin{theorem} \cite[11.3.1]{hirsch} \label{cofgen}
		Let $\mathcal M$ be a category that has all small limits and colimits. Suppose that $\mathcal M$ has a class of weak equivalences that satisfies the two-out-of-three property and that is closed under retracts. Let $I$ and $J$ be sets of maps in $\mathcal M$ that satisfy the following conditions.
		\begin{enumroman}
			\item \label{cofgen1} 
			Both $I$ and $J$ permit the small object argument
			\cite[10.5.15]{hirsch}.
			
			\item \label{cofgen2} 
			Every $J$-cofibration is an $I$-cofibration and a weak
			equivalence.
			
			\item \label{cofgen3} 
			Every $I$-injective is a $J$-injective and a weak
			equivalence.
			
			\item \label{cofgen4}
			One of the following conditions holds:
			\begin{enumerate}
				\item \label{cofgen4a}
				a map that is an $I$-cofibration and a weak equivalence is a
				$J$-cofibration, or
				
				\item \label{cofgen4b}
				a map that is both a $J$-injective and a weak equivalence is
				an $I$-injective.
			\end{enumerate}
		\end{enumroman}
		Then there is a cofibrantly generated model category structure on
		$\mathcal M$ in which $I$ is a set of generating cofibrations and
		$J$ is a set of generating acyclic cofibrations.
	\end{theorem}
	
	Finally, we consider the additional structure of cartesian model categories.  We take the following definition from \cite[2.2]{rezk}; the version there includes an additional equivalent formulation of the last condition, but we omit it as we do not need it here.
	
	\begin{definition} \label{cartdef}
		A model category $\mathcal M$ is \emph{cartesian} if the underlying category is cartesian closed, its terminal object is cofibrant, and, if $f \colon X \rightarrow X'$ and $g \colon Y \rightarrow Y'$ are cofibrations in $\mathcal M$, then the pushout-corner map
		\[ X \times Y' \amalg_{X \times Y} X' \times Y \rightarrow X' \times Y' \]
		is a cofibration that is a weak equivalence if either $f$ or $g$ is.  
	\end{definition}
	
	\section{Complete Segal spaces and Segal categories} \label{cssandsc}
	
	In this section, we recall the complete Segal space and Segal category models for $(\infty,1)$-categories. All the models for $(\infty,n)$-categories that we develop in this paper can be regarded as suitable generalizations of one or both of these models.
	
	To start, for any $k \geq 0$, consider the $k$-simplex, or representable simplicial set $\Delta[k] = \Hom_\Delta(-, [k])$.  We are interested in the inclusion of the sub-simplicial set $G[k]$, defined to be the colimit of the diagram
	\[ \xymatrix@1{\Delta[1] \ar[r]^{d_0} & \Delta[0] & \Delta[1] \ar[l]_{d_1} \ar[r]^{d_0} & \cdots & \Delta[1] \ar[l]_{d_1}} \]
	in which there are $k$ copies of $\Delta[1]$ glued together along copies of $\Delta[1]$.  We can depict $G[k]$ as a string of $k$ consecutive arrows
	\[ \bullet \rightarrow \bullet \rightarrow \cdots \rightarrow \bullet \] but with no higher simplices.  This simplicial set $G(k)$ is sometimes called the \emph{spine} of $\Delta[k]$.
	
	Now we would like to regard these simplicial sets as discrete simplicial spaces, or functors $\Deltaop \rightarrow \SSets$; given a simplicial set $K$ we denote the corresponding discrete simplicial space by $K^t$.  In particular, $(K^t)_k = K_k$ for each $k \geq 0$. In contrast, the constant simplicial space, also denoted by $K$, is given by the simplicial set $K$ as each level.  The two are ``transpose" to one another by reversing the role of the two simplicial directions, but they should be regarded as different from one another.  In particular, the simplicial spaces $\Delta[k]^t$ are representable simplicial spaces, and so for any simplicial space $X$ we have
	\[ \Map(\Delta[k]^t, X) \cong X_k. \]
	
	Given any simplicial space $X$, and any $k \geq 0$, consider the \emph{Segal map} induced by the inclusion $G[k]^t \rightarrow \Delta[k]^t$:
	\[ X_n = \Map(\Delta[k]^t, X) \rightarrow \Map(G[k]^t, X) \cong \underbrace{X_1 \times_{X_0} \cdots \times_{X_0} X_1}_k. \]
	These maps are isomorphisms for $k=0,1$, so we restrict our attention to $k \geq 2$.
	
	\begin{definition}
		A Reedy fibrant simplicial space $X$ is a \emph{Segal space} if, for all $k \geq 2$, the Segal maps are all weak equivalences of simplicial sets.
	\end{definition}
	
	The following model structure can be obtained by localizing the Reedy model structure on simplicial spaces with respect to the maps $G[k]^t \rightarrow \Delta[k]^t$ for $k \geq 2$.
	
	\begin{theorem} \cite[7.1]{rezk}
		There is a cartesian model structure $\sesp$ on the category of simplicial spaces in which the fibrant objects are precisely the Segal spaces.
	\end{theorem}
	
	\begin{remark} \label{segalisreedyfibrant}
		We follow the convention of Rezk and require that a Segal space be Reedy fibrant, so that we have the above concise description of the fibrant objects in the corresponding model structure.  Doing so additionally permits us to consider the Segal maps as we have described them, rather than taking homotopy mapping spaces, and hence a homotopy limit in the definition of Segal maps. 
	\end{remark}
	
	Segal spaces behave like categories up to homotopy, an idea that can be made precise in the following way.  We can define the set of \emph{objects} of a Segal space to be $X_{0,0}$.  Given two objects $x,y$ of $X$, the \emph{mapping space} between $x$ and $y$ in $X$ is defined as the pullback
	\[ \xymatrix{\map_X(x,y) \ar[r] \ar[d] & X_1 \ar[d]^{(d_1, d_0)} \\
		\{(x,y)\} \ar[r] & X_0 \times X_0.} \]
	Since $X$ is assumed to be Reedy fibrant, the right-hand vertical map is a fibration, so the mapping space is in fact a homotopy pullback of this diagram.
	
	The definition of a Segal space in terms of the Segal maps guarantees the existence of an up-to-homotopy composition on mapping spaces: again because we have assumed that $X$ is Reedy fibrant, the left-hand map in the diagram
	\[ \xymatrix@1{X_1 \times_{X_0} X_1 & X_2 \ar[l]_-{(d_1, d_0)} \ar[r]^{d_1} & X_1} \]
	is an acyclic fibration and hence has a section which serves as a homotopy inverse.  We can also define the \emph{homotopy category} $\Ho(X)$ of a Segal space $X$, whose objects are those of $X$ and whose morphisms are the path components of the mapping spaces of $X$.
	
	However, in the above definition we have only used the 0-simplices of $X_0$ as the objects, so in some sense the additional simplicial data of $X_0$ is extraneous.  Indeed, to get an $(\infty,1)$-category we want to have a discrete space of objects, rather than an arbitrary space as in a general Segal space.  There are two approches to remedying this situation.  We start with the simplest: requiring that $X_0$ be discrete, so that all its higher simplices are degenerate.
	
	\begin{definition}
		A \emph{Segal category} is a Segal space $X$ such that $X_0$ is discrete.
	\end{definition}
	
	\begin{remark} \label{reedyrmk}
		Note that here we have assumed that a Segal category is Reedy fibrant.  This assumption is not made in many other treatments of Segal categories, for example in \cite{thesis}.  In some situations, we want to consider Segal categories that are projective fibrant, rather than Reedy fibrant, and hence do not specify one or the other in the basic definition.  In this paper, we only consider Segal categories and generalizations that are Reedy fibrant, so to be as streamlined as possible we include this assumption here.
	\end{remark}
	
	For many homotopy-theoretic purposes, however, requiring that a space be discrete is unnatural.  An alternative approach is to require $X_0$ to be equivalent to the space of homotopy equivalences in $X_1$, as described by Rezk \cite[\S 5.7]{rezk}.  Since a Segal space has a notion of up-to-homotopy composition, as well as identity maps defined by those in the image of the degeneracy map $X_0 \rightarrow X_1$, there is a natural definition of homotopy equivalences as those maps that have an inverse up to homotopy.  They form a subspace $X_{\heq} \subseteq X_1$, and indeed comprise some of the path components of $X_1$ \cite[5.8]{rezk}. The degeneracy map $s_0 \colon X_0 \rightarrow X_1$ factors through $X_{\heq}$, allowing for the following definition.
	
	\begin{definition} 
		A Segal space $X$ is \emph{complete} if the map $s_0 \colon X_0 \rightarrow X_{\heq}$ is a weak equivalence of simplicial sets.
	\end{definition}
	
	The proof that Segal categories and complete Segal spaces are equivalent models for $(\infty,1)$-categories is given by a Quillen equivalence of appropriate model categories.  The model structure for complete Segal spaces is obtained as a further localization of the Segal space model structure $\sesp$.  We localize with respect to the map $E^t \rightarrow \Delta[0]^t$, where $E$ is the simplicial set given by the nerve of the category with two objects and a single isomorphism between them.  One can check that $\Map(E^t, X) \simeq X_{\heq}$, verifying that the local objects with respect to this map are indeed complete. 
	
	\begin{theorem} \cite[7.2]{rezk}
		There is a cartesian model structure $\css$ on the category of simplicial spaces such that the fibrant objects are the complete Segal spaces.
	\end{theorem}
	
	For Segal categories, the underlying category for the model structure is the category of \emph{Segal precategories}, or simplicial spaces $X$ with $X_0$ discrete, such that the fibrant objects are the Segal categories.  However, since this category does not admit a Reedy-like model structure with levelwise weak equivalences, we cannot obtain the Segal category model structure as a localization; see \cite[\S 3.12]{thesis} for more details.   Roughly speaking, the problem is that if we want levelwise weak equivalences and cofibrations that are monomorphisms, it is impossible to have the necessary factorizations. 
	Thus, to establish the desired model structure, we need a precise definition of the appropriate weak equivalences.
	
	To this end, let us return to the Segal space model structure for a moment.  Given any simplicial space $X$, we can take a functorial fibrant replacement of it in the model structure $\sesp$, in which the fibrant objects are the (not necessarily complete) Segal spaces.  Denoting this fibrant replacement by $L_{\Se}$, we make the following definition.
	
	\begin{definition} \label{dk}
		A map $f \colon X \rightarrow Y$ of simplicial spaces is a \emph{Dwyer-Kan equivalence} if:
		\begin{itemize}
			\item for any $x,y \in \ob(X)$, the map 
			\[ \map_{L_{\Se}X}(x,y) \rightarrow \map_{L_{\Se}Y}(fx,fy) \]
			is a weak equivalence of simplicial sets; and
			
			\item the induced functor 
			\[ \Ho(L_{\Se}X) \rightarrow \Ho(L_{\Se}Y) \]
			is essentially surjective.
		\end{itemize}
	\end{definition}
	
	This definition can be adapted to the setting of Segal precategories, using a suitable modification of the functor $L_{\Se}$ that retains the necessary discreteness condition \cite[\S 5]{thesis}; see also the arguments in Proposition \ref{discreteloc}.
	
	The importance of Dwyer-Kan equivalences is illustrated in the following results of Rezk.
	
	\begin{theorem} \cite[7.6, 7.7]{rezk} \label{dkincss}
		\begin{enumerate}
			\item \label{dkbetweensegalincss} 
			A map $X \rightarrow Y$ of Segal spaces is a Dwyer-Kan equivalence if and only if it is a weak equivalence in $\css$.
			
			\item \label{dkbetweencompleteincss} 
			A map $X \rightarrow Y$ of complete Segal spaces is a Dwyer-Kan equivalence if and only if it is a levelwise weak equivalence of simplicial sets.
		\end{enumerate}
	\end{theorem}
	
	In particular, if we want a model structure whose weak equivalences between Segal categories behave like the weak equivalences of $\css$, the Dwyer-Kan equivalences provide good candidates.
	
	\begin{theorem} \cite[5.1]{thesis}, \cite{pel}
		There is a cartesian model structure $\Secat$ on the category of Segal precategories in which the weak equivalences are the Dwyer-Kan equivalences and the fibrant objects are the Segal categories.
	\end{theorem}
	
	The common notion of Dwyer-Kan equivalence in the model structures $\css$ and $\Secat$ is key to the proof of the following theorem.
	
	\begin{theorem} \cite[6.3]{thesis} \label{secatcss}
		The inclusion functor from the category of Segal precategories into the category of all simplicial spaces has a right adjoint, and this adjunction induces a Quillen equivalence
		\[ \Secat \rightleftarrows \css. \]
	\end{theorem}
	
	The right adjoint functor $R$ serves as a discretization functor, and can be described on objects as follows.  Let $W$ be a simplicial space.  Let $U=\cosk_0(W)$ be the 0-coskeleton of $W$, and $V=U_{*,0}$, the discrete simplicial space given by the 0-simplices in each degree of $U$. Alternatively, $V = \cosk_0(W_{*,0})$, where $W_{*,0}$ denotes the discrete simplicial space consisting of the zero simplices in each degree of $W$. Then $RW$ is defined to be the pullback 
	\[ \xymatrix{RW \ar[r] \ar[d] & V \ar[d] \\
		W \ar[r]& U,} \]
	where $V \rightarrow U$ is the inclusion and $W \rightarrow U$ is the canonical map from the coskeleton.
	
	\section{$\Theta_n$-spaces} \label{thetanreview}
	
	In this section, we review the definition of $\Theta_n$-spaces, which serve as a model for higher-categorical complete Segal spaces, and summarize some of the key constructions that we need here.  
	The categories $\Theta_n$ were originally described by Joyal, using a direct definition \cite{joyal}.  Here we have chosen to use their inductive description via the $\Theta$-construction of Berger \cite{berger}, which is also described by Rezk in \cite[3.2]{rezktheta}.
	
	\begin{definition}
		Let $\mathcal C$ be a small category.  Define $\Theta \mathcal C$ to be the category with objects $[m](c_1, \ldots, c_m)$ where $[m]$ is an object of $\Delta$ and each $c_i$ is an object of $\mathcal C$.  A morphism
		\[ [q](c_1, \ldots ,c_q) \rightarrow [m](d_1, \ldots, d_m) \] is given by $(\delta, \{f_{ij}\})$ where $\delta \colon [q] \rightarrow [m]$ in $\Delta$ and $f_{ij} \colon c_i \rightarrow d_j$ are morphisms in $\mathcal C$ indexed by $1 \leq i \leq q$ and $1 \leq j \leq m$ where $\delta(i-1) < j \leq \delta (i)$.
	\end{definition}
	
	Let us use this definition to describe our categories of interest here.
	
	\begin{definition}
		Let $\Theta_0$ be the terminal category with a single object and no non-identity morphisms.  Inductively define $\Theta_n=\Theta \Theta_{n-1}$. 
	\end{definition}
	
	Observe that $\Theta_1=\Delta$.  To build some intuition about $\Theta_n$ for higher $n$, let us look more closely at $\Theta_2 = \Theta \Delta$. Its objects are of the form $[q]([c_1], \ldots, [c_q])$, where $[q]$, as well as each $[c_i]$, is an object of $\Delta$. We can think of this object as being a copy of the diagram $[q]$ whose arrows are labeled by $[c_1], \ldots, [c_q]$.  For example, the object $[4]\left([2], [3], [0], [1]\right)$ can be depicted as
	\[ \xymatrix@1{0 \ar[r]^{[2]} & 1 \ar[r]^{[3]} & 2 \ar[r]^{[0]} & 3\ar[r]^{[1]} & 4.} \]
	Since these labels themselves can be visualized as strings of arrows, we can further illustrate our object as
	
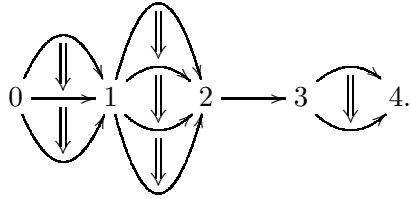
\[\xymatrix{ {0} \ar@/^2pc/[r]^{}="10"  \ar[r]^{}="11" \ar@/_2pc/[r]^{}="12" \ar@{=>}"10";"11" \ar@{=>}"11";"12"
		& {1} \ar@/^3pc/[r]^{}="20" \ar@/^1pc/[r]^{}="21"
		\ar@/_1pc/[r]^{}="22" \ar@/_3pc/[r]^{}="23" \ar@{=>}"20";"21"
		\ar@{=>}"21";"22" \ar@{=>}"22";"23"
		& {2} \ar[r]
		& {3} \ar@/^1pc/[r]^{}="30" \ar@/_1pc/[r]^{}="31" \ar@{=>}"30";"31"
		& {4.}
	}\]
	This diagram can be regarded as generating a strict 2-category by composing 1-cells and 2-cells whenever possible.  In other words, the objects of $\Theta_2$ can be seen as encoding all possible finite composites, whether horizontal or vertical, that can take place in a 2-category, much as the objects of $\Delta$ can be thought of as listing all the finite composites that can occur in an ordinary category.
	
	\begin{example} \label{onenotation}
		Of key importance in this paper are the objects of $\Theta_n$ given by a single morphism, or free-standing ``cell", of each dimension up to $n$.  For example, we have a single object $\bullet$, a single 1-cell $\bullet \rightarrow \bullet$, and a single 2-cell
		
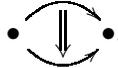
\[\xymatrix{{\bullet} \ar@/^1pc/[r]^{}="30" \ar@/_1pc/[r]^{}="31" \ar@{=>}"30";"31"
			& {\bullet.} }\]
		These objects are denoted by $[0]$, $[1]([0])$, and $[1]([1])$ as objects of $\Theta_2$.  More generally, in $\Theta_n$ we have objects $[1]([1]( \cdots ([0])) \cdots)$
		where we have up to $n-1$ occurrences of $[1]$ concluding with a $[0]$, and finally the object 
		\[ [1]([1](\cdots [1]) \cdots). \]
		These objects can be depicted as the free-standing cells of increasing dimension, starting with dimension 0 (an ``object") and going up through dimension $n$ (an ``$n$-morphism" or ``$n$-cell"). 
		
		To simplify the notation, we often write $[1]^{(0)}$ for $[0]$, and then $[1]^{(i)}=[1]([1]^{(i-1)})$ for any $0<i<n$.  
		Observe that $[0]$ is the terminal object in any $\Theta_n$, and likewise that $[1]^{(i)}$ can be considered as an object of $\Theta_n$ for any $n >i$.  While there is some ambiguity here about what $n$ is, the idea is simply that we have $i$-morphisms in any $n$-category for $n>i$.
		
		We include in our notation the free-standing $n$-cell, which corresponds to the object $[1]([1](\cdots [1]) \cdots)$ of $\Theta_n$, and denote this object by $[1]^{(n)}$.  This notation is in potential conflict with the above, if we move from $\Theta_n$ to $\Theta_k$ for some $k>n$, but the context should make this distinction, especially since the essential shape remains the same.  Indeed, this object is simply given by 
		\[ \underbrace{[1]([1](\cdots [1]}_n([0])) \cdots) \] 
		in $\Theta_k$ when $k>n$.
	\end{example}
	
	Let us consider $\Theta_n$-\emph{sets}, which are functors $\Thetanop \rightarrow \Sets$. For any object $[q](c_1, \ldots, c_q)$, let $\Theta[q](c_1, \ldots, c_q)$ denote the representable functor $\Hom_{\Theta_n}(-,[q](c_1, \ldots, c_q))$.
	
	Since we want to generalize complete Segal spaces, we are more interested in functors $\Thetanop \rightarrow \SSets$.  Observe that any simplicial set can be regarded as a constant functor of this kind, and any functor $\Thetanop \rightarrow \Sets$, in particular the representable functors just described, can be regarded as levelwise discrete functors to $\SSets$.  
	
	The category $\Thetanop$ is a Reedy category, with the degree of an object $[m](c_1, \ldots, c_m)$ defined inductively as the sum of $m$ with the degrees of the objects $c_i$ in $\Theta_{n-1}$ \cite[3.14]{berger}.  Hence, the functor category $\SSets^{\Thetanop}$ can be equipped with the Reedy model structure, which we prove in \cite[3.10]{elegant} agrees with the injective model structure.  In particular, the cofibrations are the levelwise monomorphisms of simplicial sets, and every object is cofibrant.  
	
	We recall from \cite[\S 6.1]{inftyn1} that a set of generating cofibrations of this model structure are given by
	\[ \partial \Delta[m] \times \Theta[q](c_1, \ldots, c_q) \cup \Delta[m] \times \partial \Theta[q](c_1, \ldots, c_q) \rightarrow \Delta[m] \times \Theta[q](c_1, \ldots, c_q) \]
	where $m, q \geq 0$ and $c_1, \ldots, c_q \in \ob(\Theta_{n_1})$, and where $\partial \Theta[q](c_1, \ldots, c_q)$ denotes the boundary of the representable object $\Theta[q](c_1, \ldots, c_q)$, defined by mapping out of objects of strictly lower degree than that of $[q](c_1, \ldots, c_q)$.  The domain of this map is the union along the intersection of the two spaces, $\partial \Delta[m] \times \partial \Theta[q](c_1, \ldots, c_q)$, but here and in what follows we omit this additional notation for the sake of brevity.  Note that the simplicial sets and $\Theta_n$-sets appearing in this definition are appropriately constant in the $\Theta_n$ and simplicial directions, respectively. A set of generating acyclic cofibrations can be defined similarly, replacing the boundaries $\partial \Delta[m]$ with horns $\Lambda^k[m]$ for $m \geq 1$ and $0 \leq k \leq m$.
	
	To obtain models for $(\infty,n)$-categories, we want to ask for appropriate Segal maps to be weak equivalences, and so we generalize the development of Segal spaces as follows.
	
	Given $q \geq 2$ and $c_1, \ldots, c_q$ objects of $\Theta_{n-1}$, define the object
	\[ G[m](c_1, \ldots, c_m)= \colim (\Theta[1](c_1) \leftarrow \Theta[0] \rightarrow \cdots \leftarrow \Theta[0] \rightarrow \Theta[1](c_m)). \]
	There is an inclusion map
	\[\se^{(c_1, \ldots, c_q)} \colon G[q](c_1, \ldots, c_q) \rightarrow \Theta[q](c_1, \ldots, c_q). \]  Now define the set
	\[ \Se_{\Theta_n} = \{ \se^{(c_1, \ldots, c_q)} \mid q \geq 2, c_1, \ldots c_q \in \ob(\Theta_{n-1})\}. \]
	
	\begin{example}
		Referring to the example of the object
		\[ [4]\left([2], [3], [0], [1]\right) \] of $\Theta_2$ above, we have the representable functor
		\[ \Theta[4]\left([2], [3], [0], [1]\right). \]  
		The corresponding functor
		\[ G[4]\left([2], [3], [0], [1]\right) \] 
		consists of the union of the representable functors $\Theta[1]([2])$, $\Theta[1]([3])$, $\Theta[1]([0])$, and $\Theta[1]([1])$, glued together along the representable functors corresponding to the intersection points, each given by $\Theta[0]$.  
		If we localize with respect to the inclusion
		\[ G[4]\left([2], [3], [0], [1]\right) \rightarrow \Theta[4]\left([2], [3], [0], [1]\right), \]
		then an object $X$ is local if having these vertical composites guarantees the existence of all horizontal composites of 1-cells and 2-cells.  Namely, the induced map
		\[ X[4]([2], [3], [0], [1]) \rightarrow X[1]([2]) \times_{X[0]} X[1]([3]) \times_{X[0]} X[1]([0]) \times_{X[0]} X[1]([1]) \]
		is a weak equivalence of simplicial sets.
		
		However, such a localization only gives us horizontal composition, not vertical composition.  For example, we also want the map
		\[ X[1]([2]) \rightarrow X[1]([1]) \times_{X[1]([0])} X[1]([1]) \]
		to be a weak equivalence of simplicial sets.   
	\end{example}
	
	The previous example illustrates that being local with respect to the maps in $Se_{\Theta_n}$ is not sufficient when $n>1$, as it only gives an up-to-homotopy composition horizontally.  Encoding other levels of composition is achieved inductively, making use of the intertwining functor $V[1] \colon \SSets^{\Theta_{n-1}^{op}} \rightarrow \SSets^{\Thetanop}$ to translate a set $\mathcal S$ of maps in $\SSets^{\Theta_{n-1}^{\op}}$ into a set $V[1](\mathcal S)$ of maps in $\SSets^{\Thetanop}$.  Let us briefly recall this functor; full details can be found in \cite[\S 4.4]{rezktheta}.  
	
	Given a functor $A \colon \Theta_{n-1}^{\op} \rightarrow \SSets$, define $V[1](A) \colon \Thetanop \rightarrow \SSets$ by
	\[ [q](c_1, \ldots, c_q) \mapsto \coprod_{\delta \colon [q] \to [1]} \prod_{i=1}^q A(c_i). \]
	The idea is that $V[1](A)$ models a category enriched in $\SSets^{\Theta_{n-1}^{\op}}$, with two objects $x$ and $y$ and one nontrivial mapping object from $x$ to $y$ given by $A$.  The mapping object from $y$ to $x$ is empty, and the mapping objects at $x$ and $y$ each consist of an identity morphism only.
	
	Let $\mathcal S_1=\Se_{\bf \Delta} = \{G(n)^t \rightarrow \Delta[n]^t \mid n \geq 2\}$, and for $n \geq 2$, inductively define $\mathcal S_n=\Se_{\Theta_n} \cup V[1](\mathcal S_{n-1})$.  Thus, in $\SSets^{\Thetanop}$, we have $n$ different Segal conditions, corresponding to the desired composition in each of the $n$ categorical levels.
	
	\begin{theorem} \cite[8.5]{rezktheta}
		Localizing the Reedy model structure $\SSets^{\Thetanop}$ with respect to $\mathcal S_n$ results in a cartesian model category whose fibrant objects are higher-order analogues of Segal spaces.
	\end{theorem}
	
	We denote this model structure by $\Theta_n\sesp$ and refer to its fibrant objects as \emph{$\Theta_n$-Segal spaces}.
	
	However, to get models for $(\infty,n)$-categories, we want to incorporate higher-order completeness conditions as well.  Again, we localize with respect to some maps, and to do so, we make use of the underlying simplicial space of a functor $\Thetanop \rightarrow \SSets$.  
	
	Consider the functor $\tau_\Theta \colon {\bf \Delta} \rightarrow \Theta_n$ defined by
	\[ \tau_\Theta[k] = [k]([0], \ldots, [0]), \]  
	which by \cite[4.1]{rezktheta} induces a Quillen pair on Reedy model structures
	\begin{equation} \label{underlying}
		(\tau_\Theta)_\# \colon \SSets^{\Deltaop} \rightleftarrows \SSets^{\Thetanop} \colon \tau_\Theta^*.
	\end{equation} 
	The functor $\tau_\Theta^*$ takes a functor $X \colon \Thetanop \rightarrow \SSets$ to its \emph{underlying simplicial space}.  Recalling that complete Segal spaces are defined by localizing with respect to 
	\[ \Cpt_{\bf \Delta}= \{E^t \rightarrow \Delta[0]^t\}, \] 
	for $n \geq 2$ we use the left adjoint functor $(\tau_\Theta)_\#$ to define 
	\[ \Cpt_{\Theta_n}=\{(\tau_\Theta)_\# E^t \rightarrow (\tau_\Theta)_\# \Delta[0]^t\}. \]  
	Just as for the Segal conditions, localizing with respect to this map only encodes one completeness condition.  Specifically, an object $X \colon \Thetanop \rightarrow \SSets$ is local with respect to this map when
	\[ X[0] \simeq X[1]([0])_{\heq}, \]
	where the space on the right-hand side is a suitable space of homotopy equivalences.  It can be defined directly in a way similar to the definition for complete Segal spaces, or more formally as $\Map((\tau_\Theta)_\# E, X)$.  To capture the other necessary completeness conditions, namely that
	\[ X[1]^{(i)} \simeq X[1]^{(i+1)}_{\heq} \]
	for $0<i<n$, we use the intertwining functor.  For example, when $i=2$, we can define 
	\[ X[1]^{(2)}_{\heq} = \Map(V[1]((\tau_\Theta)_\# E^t), X). \]
	
	Combining with the Segal maps, let 
	\[ \mathcal T_1=\Se_{\Delta} \cup \Cpt_{\Delta} \] 
	and, for $n \geq 2$,
	\[ \mathcal T_n=\Se_{\Theta_n} \cup \Cpt_{\Theta_n} \cup V[1](\mathcal T_{n-1}), \]
	again recalling that $\Theta_1=\Delta$.
	
	\begin{theorem} \cite[8.1]{rezktheta} \label{thetanmc}
		Localizing $\SSets^{\Thetanop}$ with respect to the set $\mathcal T_n$ results in a cartesian model category that we denote by $\Theta_n\css$.
	\end{theorem}
	
	We would like to have a good description of the fibrant objects in this model structure.  To this end, we first define mapping objects.
	
	\begin{definition} \label{thetamapping}
		Given a functor $X \colon \Thetanop \rightarrow \SSets$ and any $(x_0, x_1) \in X[0]_0 \times X[0]_0$, we define the \emph{mapping object} $M_X^\Theta(x_0, x_1) \colon \Theta_{n-1}^{\op} \rightarrow \SSets$, evaluated at any object $c$ of $\Theta_{n-1}$, as the pullback of the diagram
		\[ \{(x_0, x_1)\} \rightarrow X[0]\times X[0] \leftarrow X[1](c). \]
	\end{definition}
	
	Revisiting the adjunction \eqref{underlying}, the functor
	\[ \tau_\Theta^* \colon \SSets^{\Thetanop} \rightarrow \SSets^{\Deltaop} \]
	is given by $(\tau_\Theta^*X)_m = \Theta[n]([0], \ldots, [0])$, where here $[0]$ is the terminal object of $\Theta_{n-1}$.
	
	We have the following explicit description of the fibrant objects of the model structure $\Theta_n\css$ that we use, for example, in \cite{inftyn2}.
	
	\begin{definition} \label{thetanspace}
		A \emph{$\Theta_n$-space} is a functor $X \colon \Thetanop \rightarrow \SSets$ such that:
		\begin{enumerate}
			\item $X$ is Reedy fibrant;
			
			\item for every $m \geq 2$ and $c_1, \ldots, c_m \in \ob(\Theta_{n-1})$ the Segal map
			\[ X[m](c_1, \ldots, c_m) \rightarrow X[1](c_1) \times_{X[0]} \cdots \times_{X[0]} X[1](c_m) \]
			is a weak equivalence of simplicial sets;
			
			\item the underlying simplicial space $\tau_\Theta^* X$ is a complete Segal space; and
			
			\item for every $(x_0, x_1) \in X[0]_0 \times X[0]_0$, the mapping object $M_X^\Theta(x_0, x_1)$ is a $\Theta_{n-1}$-space.
		\end{enumerate}
	\end{definition}
	
	\begin{remark}
		We have chosen to follow Rezk's original terminology and refer to these objects as ``$\Theta_n$-spaces".  It is arguably more accurate to call them ``complete Segal $\Theta_n$-spaces", a convention we adopt in \cite{survey}. For the purposes of this paper, however, this specification leads to unwieldy terminology; to avoid having to refer repeatedly to atrocities such as `` complete Segal objects in complete Segal $\Theta_n$-spaces" in what follows, we have chosen to revert back to the more concise name.  We retain the specification of ``$\Theta_n$-Segal space" when we refer to an object that satisfies the Segal condition but no completeness assumptions.
	\end{remark}
	
	We conclude this section by recalling two different ways to think of a homotopy category of a $\Theta_n$-Segal space.
	
	\begin{definition}
		Let $X$ be a $\Theta_n$-Segal space.  Then its \emph{enriched homotopy category} $\uHo(X)$ has object set $X[0]$ and mapping object
		\[\uMap_{\uHo(X)}(x_0, x_1) = M_X^\Theta(x_0, x_1) \colon \Theta_{n-1}^{\op} \rightarrow \SSets. \]  
		Its (ordinary) \emph{homotopy category} $\Ho^\Theta(X)$ has the same objects but has 
		\[ \Hom_{\Ho^\Theta(X)}(x_0, x_1) = \Ho_{\tau_\Theta^*X}(x_0, x_1). \]
	\end{definition}
	
	Alternatively described, $\Ho^\Theta(X)$ is the homotopy category of the underlying Segal space of $X$, in the sense described in Section \ref{cssandsc}.
	
	We would like to use these definitions to generalize Definition \ref{dk} to a notion of Dwyer-Kan equivalence for $\Theta_n$-spaces, and then prove the analogue of Theorem \ref{dkincss} in this context.  Because we have need of them in the proof, we first take a detour to review complete Segal objects in $\Theta_n$-spaces and the notion of Dwyer-Kan equivalence in that context.
	
	\section{Segal and complete Segal objects in $\Theta_n$-spaces}
	
	One feature of $\Theta_n$-spaces is that they are suitably equivalent to categories enriched in $\Theta_{n-1}$-spaces, following the general principle that $(\infty,n)$-categories should be equivalent to categories enriched in $(\infty,n-1)$-categories.  One way to model categories weakly enriched in $\Theta_{n-1}$-spaces is via the structure of a complete Segal object in $\Theta_{n-1}$-spaces.  We give a brief review here, and refer the reader to \cite{inftyn2} for more details.  
	
	The main idea is that, just as a complete Segal space can be thought of as a category weakly enriched in spaces and is given by a functor $W \colon \Deltaop \rightarrow \SSets$, we can describe a complete Segal object in $\Theta_n$-spaces as a functor $W \colon \Deltaop \rightarrow \Theta_n\css$. We emphasize the model structure $\Theta_n\css$ here because it determines the weak equivalences we use for our Segal conditions, but the objects of the underlying category are functors $W \colon \Deltaop \rightarrow \SSets^{\Thetanop}$.
	
	As with complete Segal spaces and $\Theta_n$-spaces, it is helpful to look first at objects that satisfy only the relevant Segal condition.  Our first approach uses a straightforward generalization of the definition of Segal space.
	
	\begin{definition}
		A \emph{Segal object in $\Theta_n$-spaces} is a Reedy fibrant functor $W \colon \Deltaop \rightarrow \SSets^{\Thetanop}$ such that, for every $m \geq 2$, the Segal map
		\[ W_m \rightarrow \underbrace{W_1 \times_{W_0} \cdots \times_{W_0} W_1}_m \]
		is a weak equivalence in the model structure $\Theta_n\css$.  
	\end{definition}
	
	It can be helpful here to think of such functors as $W \colon \Deltaop \rightarrow \Theta_n\css$ to emphasize the model structure on the target category.  We often refer to Segal objects in $\Theta_n$-spaces simply as \emph{Segal objects} for simplicity.
	
	However, there is an equivalent definition that is more widely used in the literature, and that enables a cleaner description of the completeness condition.  Here it is helpful to regard functors $W \colon \Deltaop \rightarrow \SSets^{\Thetanop}$ instead as functors $W \colon \Deltaop \times \Thetanop \rightarrow \SSets$.
	
	For our alternate definition, which is closely related to Definition \ref{thetanspace}, we need a notion of mapping object that is analogous to the one given in Definition \ref{thetamapping} for $\Theta_n$-spaces.
	
	\begin{definition}
		Given a functor $W \colon \Deltaop \times \Thetanop \rightarrow \SSets$ and any $x_0, x_1 \in W([0], [0])_0$, the \emph{mapping object} $M^\Delta_W(x_0, x_1) \colon \Thetanop \rightarrow \SSets$ is defined levelwise by pullbacks
		\begin{equation} \label{mappingobj}
			\xymatrix{M^\Delta_W(x_0, x_1)(c) \ar[r] \ar[d] & W([1],c) \ar[d] \\
				\{(x_0, x_1)\} \ar[r] & W([0], c) \times W([0],c).} 
		\end{equation}
	\end{definition}
	
	The following result is known to experts, but we are not aware of a proof in the literature, so we include one here.  It is of interest in part due to the subtle role of the two different Reedy structures involved.  Recall that a $\Theta_n$-space $X$ is \emph{essentially constant} if for any object $[m](c_1, \ldots, c_m)$ of $\Theta_n$ the unique map from it to $[0]$ induces a weak equivalence $X[0] \rightarrow X[m](c_1, \ldots, c_m)$.
	
	\begin{prop} \label{altsegalobj}
		A Reedy fibrant functor $W \colon \Deltaop \times \Thetanop \rightarrow \SSets$ is a Segal object in $\Theta_n$-spaces with $W_0$ essentially constant if and only if the following conditions hold:
		\begin{enumerate}
			\item \label{segal1}
			for any $m \geq 2$ and $c \in \ob(\Theta_n)$, the Segal map
			\[ W([m], c) \rightarrow W([1],c) \times_{W([0], c)} \cdots \times_{W([0],c)} W([1],c) \]
			is a weak equivalence of simplicial sets; and
			
			\item \label{segal2}
			for any $x_0, x_1 \in W([0],[0])_0$, the mapping object $M^\Delta_W(x_0, x_1)$ is a $\Theta_n$-space.
		\end{enumerate}
	\end{prop}
	
	\begin{proof}
		Suppose that $W$ is a Segal object in $\Theta_n$-spaces, so for each $m \geq 2$ the map
		\[ W_m \rightarrow W_1 \times_{W_0} \cdots \times_{W_0} W_1 \]
		is a weak equivalence in $\Theta_n\css$.  Since $W$ is assumed to be Reedy fibrant, $W_m$ is a $\Theta_n$-space for each $m \geq 0$ \cite[15.3.12]{hirsch}.  Since $\Theta_n\css$ is obtained as a localized model category, and local weak equivalences between fibrant objects are levelwise weak equivalences, each Segal map above is a levelwise weak equivalence of functors $\Thetanop \rightarrow \SSets$, i.e., the maps as in \eqref{segal1} are weak equivalences of simplicial sets.
		
		To check \eqref{segal2}, consider $M_W^\Delta(x,y)$ for fixed $x,y \in W([0],[0])_0$.  Since $W$ is assumed to be Reedy fibrant, the right vertical map in \eqref{mappingobj} is a fibration between $\Theta_n$-spaces, which are the fibrant objects in $\Theta_n\css$.  Since the discrete object $\{(x,y)\}$ is also a fibrant object in $\Theta_n\css$, the pullback must be as well.  It follows that $M^\Delta_W(x,y)$ is fibrant, namely, a $\Theta_n$-space.  
		
		Conversely, suppose $W$ satisfies conditions \eqref{segal1} and \eqref{segal2}.  We first want to show that $W$ is Reedy fibrant as a functor $W \colon \Deltaop \rightarrow \Theta_n\css$.  For any $m \geq 0$, let $M_mW$ denote the $m$-th matching object of $W$; using the definition of Reedy fibration \cite[15.3.3]{hirsch}, we need to show that the map $W_m \rightarrow M_mW$ is a fibration in $\Theta_n\css$.  
		
		Observe that $W_m = \uMap(\Delta[m], W)$, the functor $\Thetanop \rightarrow \SSets$ defined by 
		\[ [p](c_1, \ldots, c_p) \mapsto W([m], [p](c_1, \ldots, c_p)). \]  
		Similarly, $M_mW = \uMap(\partial \Delta[m], W)$.  Using the inclusion $\partial \Delta[m] \rightarrow \Delta[m]$, one can check that the map $W_m \rightarrow M_mW$ is indeed a Reedy fibration.  It remains to show it is a fibration in $\Theta_n\css$, for which it suffices by \cite[15.3.13]{hirsch} to show that $W_n$ is fibrant, i.e, a $\Theta_n$-space. 
		
		We apply the right adjoint $R$ to the inclusion functor of Segal precategory objects, or functors $\Deltaop \rightarrow \Thetansp$ with discrete space in degree zero, into all simplicial objects in $\Thetansp$, where $RW$ is the pullback
		\[ \xymatrix{RW \ar[r] \ar[d] & \cosk_0(W([0],[0])) \ar[d] \\
			W \ar[r] & \cosk_0(W_0).} \]
		The essential constancy of $W_0$ implies that $\cosk_0(W([0],[0])$ is levelwise weakly equivalent to $\cosk_0(W_0)$, and hence $RW \rightarrow W$ is also a levelwise weak equivalence.  But 
		\[ (RW)_1 = \coprod_{(x,y)}\map_W(x,y), \]
		which is a $\Theta_n$-space by assumption, thus $W_1$ must be as well; a similar argument can be used for $n \geq 1$.
		
		Finally, we need to check that for any $m \geq 2$ the Segal map
		\[ W_m \rightarrow W_1 \times_{W_0} \cdots \times_{W_0} W_1 \]
		is a weak equivalence in $\Theta_n\css$.  We know by assumption that for any $m \geq 2$ and any object $c$ of $\Thetanop$, the map
		\[ W([m], c) \rightarrow W([1],c) \times_{W([0],c)} \cdots \times_{(W([0],c)} W([1],c) \]
		is a weak equivalence of simplicial sets.  It follows that the Segal map above is a levelwise weak equivalence of simplicial sets, hence also a weak equivalence in $\Theta_n\css$.
	\end{proof}
	
	Observe that the proof of the forward direction did not use the essential constancy condition, but it was necessary to prove the converse.
	
	We now incorporate the completeness condition by modifying this second definition of Segal objects.  Analogously to the setting of $\Theta_n$-spaces, we make use of an underlying simplicial space functor that we define as follows.  Let $\tau_\Delta \colon \Delta \rightarrow \Delta \times \Theta_n$ be given by $[k] \mapsto ([k], [0])$.  The desired functor is the induced map
	\[ \tau_\Delta^* \colon \SSets^{\Deltaop \times \Thetanop} \rightarrow \SSets^{\Deltaop}. \] 
	
	\begin{definition} 
		A Reedy fibrant functor $W \colon \Deltaop \times \Thetanop \rightarrow \SSets$ is a \emph{complete Segal object in $\Theta_n$-spaces} if:
		\begin{enumerate}
			\item \label{segal} 
			for all $m \geq 2$ and $c \in \ob(\Theta_n)$, the map
			\[ W([m],c) \rightarrow \underbrace{W([1],c)\times_{W([0],c)} \cdots \times_{W([0],c)}  W([1],c)}_m \]
			is a weak equivalence of spaces;
			
			\item \label{localmapping} 
			for all $x_0,x_1\in W([0],[0])$, the functor $M_W^\Delta(x_0,x_1) \colon \Thetanop \rightarrow \SSets$ is a $\Theta_n$-space;
			
			\item \label{complete0} 
			the underlying simplicial space $\tau_\Delta^* W$ is a complete Segal space; and
			
			\item \label{essconst} 
			for all objects $c \in \Theta_n$, the map $W([0],[0]) \rightarrow W([0],c)$ is a weak equivalence.
		\end{enumerate}
	\end{definition}
	
	Again, we typically refer to these objects simply as \emph{complete Segal objects}.  
	
	\begin{theorem} \cite[5.9]{inftyn2}
		There is a model structure $\css(\Thetansp)$ on the category of functors $\Deltaop \times \Thetanop \rightarrow \SSets$ in which the fibrant objects are the complete Segal objects, obtained by a localization of the Reedy model structure.
	\end{theorem}
	
	Using the mapping objects $M^\Delta_W(x,y)$, we can generalize Definition \ref{dk} as follows.  Here, we define the \emph{homotopy category} of a Segal object $W$ to be the homotopy category of the underlying Segal space $\tau_\Delta^*W$, and denote it by $\Ho^\Delta(W)$.
	
	\begin{definition}
		A map $f \colon W \rightarrow Z$ of Segal objects in $\Thetansp$ is a \emph{Dwyer-Kan equivalence} if:
		\begin{itemize}
			\item for any objects $x$ and $y$ of $W$, the induced map $M_W^\Delta(x,y) \rightarrow M_Z^\Delta(fx,fy)$ is a weak equivalence in $\Thetansp$, and
			
			\item the induced functor $\Ho^\Delta(W) \rightarrow \Ho^\Delta(Z)$ is essentially surjective.
		\end{itemize}
	\end{definition}
	
	The following theorem is a generalization of Theorem \ref{dkincss}\eqref{dkbetweensegalincss}.
	
	\begin{theorem}  \cite[8.18]{inftyn2}
		A map $f \colon U \rightarrow V$ of Segal objects is a Dwyer-Kan equivalence if and only if it is a weak equivalence in the model category $\css(\Thetansp)$.
	\end{theorem}
	
	\section{Dwyer-Kan equivalences for $\Theta_n$-spaces}
	
	Of key importance in the theory of complete Segal spaces and their relationship with Segal categories are the Dwyer-Kan equivalences, which mirror the natural weak equivalences of simplicial categories that share the same name.  The main idea is to generalize the notion of equivalence of categories, namely being fully faithful and essentially surjective, to a more general context.
	
	We would like to have a similar notion for $\Theta_n$-spaces.  While the appropriate definition was given in by Rezk in \cite{rezktheta}, trying to establish a result analogous to Theorem \ref{dkincss} presented some technical difficulties.  In this section, we establish these properties, avoiding some of the technical obstacles by using the Quillen equivalence between $\Theta_n$-spaces and complete Segal objects in $\Theta_{n-1}$-spaces from \cite[7.2]{inftyn2}.
	
	The following definitions for $\Theta_n$-Segal spaces are given in \cite{rezktheta}.  We start with homotopically fully faithful maps, which make use of the mapping objects described in the previous section.
	
	\begin{definition}
		Let $X$ and $Y$ be $\Theta_n$-Segal spaces. A morphism $f \colon X \rightarrow Y$ is \emph{homotopically fully faithful} if for every $x_0, x_1 \in X[0]$ and every $c \in \ob(\Theta_{n-1})$ the map
		\[ M_X^\Theta(x_0, x_1)(c) \rightarrow M_X^\Theta(fx_0, fx_1)(c) \]
		is a weak equivalence in $\SSets$.
	\end{definition}
	
	We define essential surjectivity in terms of the homotopy category.
	
	\begin{definition}
		Let $X$ and $Y$ be $\Theta_n$-Segal spaces.  A morphism $X \rightarrow Y$ is \emph{essentially surjective} if $\Ho^\Theta(f) \colon \Ho^\Theta(X) \rightarrow \Ho^\Theta(Y)$ is an essentially surjective functor of categories.
	\end{definition}
	
	We want to consider these notions for more general functors $\Thetanop \rightarrow \SSets$, which we can accomplish via a localization functor. Let us denote by $L_{\Se}X$ the functorial localization of $X$ in the model structure $\Theta_n\sesp$.
	
	\begin{definition}
		Suppose $X, Y \colon \Thetanop \rightarrow \SSets$.  A map $X \rightarrow Y$ is a \emph{Dwyer-Kan equivalence} if the associated map $L_{\Se}X \rightarrow L_{\Se}Y$ is homotopically fully faithful and essentially surjective.  
	\end{definition}
	
	The following result is the analogue of Theorem \ref{dkincss}.
	
	\begin{theorem} \label{dkwkequiv}
		Let $X, Y \colon \Thetanop \rightarrow \SSets$ be $\Theta_n$-Segal spaces.  A map $X \rightarrow Y$ is a Dwyer-Kan equivalence if and only if it is a weak equivalence in $\Theta_n\css$.
	\end{theorem}
	
	Proving this theorem using the same strategy as the analogous result for complete Segal spaces \cite[7.7]{rezk} seems challenging, although of interest for the constructions that would need to be made along the way.  In that case, Rezk gives an explicit description of a fibrant replacement functor of a Segal space via a Dwyer-Kan equivalence.  The higher categorical version of this construction seems quite difficult to produce, although it is being investigated for $n=2$ by Miika Tuominen.  However, we can prove the above theorem more efficiently, using the fact that the analogous result is true in the context of complete Segal objects in $\Theta_{n-1}\css$. 
	
	Consider the functor $d \colon \Delta \times \Theta_{n-1} \rightarrow \Theta_n$, given by $([m],c) \mapsto [m](c,\ldots, c)$.  It induces the functor 
	\[ \begin{aligned}
		d^* \colon \SSets^{\Thetanop} & \rightarrow \SSets^{\Deltaop \times \Theta_{n-1}^{\op}} \\
		X & \mapsto \left(([m], c) \mapsto X[m](c, \ldots, c) \right) 
	\end{aligned} \]
	which has a right adjoint $d_*$ given by right Kan extension.
	
	\begin{theorem} \cite[7.1]{inftyn2}
		The adjoint pair $(d^*, d_*)$ induces a Quillen equivalence
		\[ d^* \colon \Thetansp \rightleftarrows \css(\Theta_{n-1}\css) \colon d_*. \]
	\end{theorem}
	
	In these two equivalent model structures, we have respective notions of Dwyer-Kan equivalence.  The functor $d^*$ is well-behaved with respect to the two, in the following sense.
	
	\begin{prop} 
		A map $f \colon X \rightarrow Y$ of $\Theta_n$-Segal spaces is a Dwyer-Kan equivalence if and only if $d^*f \colon d^*X \rightarrow d^*Y$ is a Dwyer-Kan equivalence in $\css(\Theta_{n-1}\css)$.
	\end{prop}
	
	\begin{proof} 
		We prove in \cite[6.3]{inftyn2} that the functor $d^*$ preserves fibrant objects; looking at the appropriate parts of that proof shows that it takes $\Theta_n$-Segal spaces to Segal objects in $\Theta_{n-1}\css$.  Thus, since we have assumed that $X$ and $Y$ are fibrant, $d^*f \colon d^*X \rightarrow d^*Y$ is a Dwyer-Kan equivalence if and only if it is homotopically fully faithful and essentially surjective, without having to apply the localization functor $L_{\Se}$.
		
		First, observe that $\Ho^\Theta(X) \cong \Ho^\Delta(d^*X)$, since both are defined in terms of the underlying Segal space of $X$.  Therefore $\Ho^\Theta(X) \rightarrow \Ho^\Delta(Y)$ is essentially surjective precisely when $\Ho^\Delta(d^*X) \rightarrow \Ho^\Delta(d^*Y)$ is.
		
		In \cite[3.10]{inftyn2} we prove that there is a natural isomorphism of mapping objects
		\[ M^\Theta_X(x_0, x_1)(c) \cong M^\Delta_{d^*X}(x_0,x_1)(c) \]
		for any $x_0, x_1 \in X[0] = (d^*X)_0$ and $c \in \ob(\Theta_{n-1})$.  It follows that the map
		\[ M^\Theta_X(x_0, x_1)(c) \rightarrow M^\Theta_Y(fx_0,fx_1)(c) \]
		is a weak equivalence in $\Theta_{n-1}\css$ if and only if 
		\[ M^\Delta_{d^*X}(x_0, x_1)(c) \rightarrow M^\Delta_{d^*Y}(fx_0,fx_1)(c) \]
		is.  
	\end{proof}
	
	\begin{proof}[Proof of Theorem \ref{dkwkequiv}]
		By the previous proposition, we know that $f \colon X \rightarrow Y$ is a Dwyer-Kan equivalence in $\Thetansp$ if and only if $d^*(f) \colon d^*X \rightarrow d^*Y$ is a Dwyer-Kan equivalence in $\css(\Theta_{n-1}\css)$.  Theorem \ref{dkwkequiv} says that the latter statement is true if and only if $d^*X \rightarrow d^*Y$ is a weak equivalence in $\css(\Theta_{n-1}\css)$.  Thus, it suffices to prove that $f \colon X \rightarrow Y$ is a weak equivalence in $\Thetansp$ if and only if $d^*f \colon d^*X \rightarrow d^*Y$ is a weak equivalence in $\css(\Theta_{n-1}\css)$.
		
		First, suppose that $f$ is a weak equivalence in $\Thetansp$.  By definition of weak equivalences in a localized model structure, and using the fact that all objects are cofibrant in $\css(\Theta_{n-1}\css)$, it suffices to show that
		\[ \Map(d^*Y,Z) \rightarrow \Map(d^*X, Z) \]
		is a weak equivalence of simplicial sets for every complete Segal object $Z$.  Using the adjunction $(d^*, d_*)$, we can consider instead the map of simplicial sets
		\[ \Map(Y, d_*Z) \rightarrow \Map(X, d_*Z). \]
		We proved in \cite[6.1]{inftyn2} that if $Z$ is a complete Segal object, then $d_*Z$ is a $\Theta_n$-space.  Since we assumed that $f \colon X \rightarrow Y$ is a weak equivalence in $\Thetansp$, this map is a weak equivalence of simplicial sets, as we needed to show.
		
		Conversely, suppose that $d_*f$ is a weak equivalence in $\css(\Theta_{n-1}\css)$, so for any complete Segal object $Z$, we have a weak equivalence of simplicial sets
		\[ \Map(d^*Y, Z) \rightarrow \Map(d^*X,Z). \]
		We need to show that 
		\[ \Map(Y,W) \rightarrow \Map(X, W) \]
		is a weak equivalence for any $\Theta_n$-space $W$.  Again using the adjunction $(d^*, d_*)$, it suffices to show that any $\Theta_n$-space $W$ can be obtained as $d_*Z$ for some complete Segal object $Z$.  
		Again using the fact that $d_*$ takes complete Segal objects to $\Theta_n$-spaces, define $Z$ by $Z([1],c)= W[1](c)$ for any $c \in \ob(\Theta_{n-1})$; the rest of the structure is thus determined by the Segal and completeness conditions.  In particular, $Z([0],c)=W[0]$ for any $c \in \ob(\Theta_{n-1})$.  Since the functor $d_*$ is defined via right Kan extension, and we are applying it to a Segal object, we obtain 
		\[ \begin{aligned}
			(d_*Z)[1](c) & = \lim_{[1](c) \to [p](b, \ldots b)} Z([p], b) \\
			& \simeq \lim_{[1](c) \to [p](b, \ldots b)} \left(Z([1],b) \times_{Z([1], b)} \cdots \times_{Z([0],b)} Z([1], b) \right) \\
			& = \lim_{[1](c) \to [p](b, \ldots b)} \left(W[1](b) \times_{W[0]} \cdots \times_{W[0]} W[1](b) \right) \\
			& \simeq \lim_{[1](c) \to [p](b, \ldots b)} W[p](b,\ldots, b) \\
			& = W[1](c).
		\end{aligned} \]
		Thus $d_*Z=W$, as we wished to show.
	\end{proof}
	
	Now we can prove the following characterization of Dwyer-Kan equivalences between $\Theta_n$-spaces, which is a generalization of Theorem \ref{dkincss}\eqref{dkbetweencompleteincss}.  In the proof, we make use of the objects $[1]^{(i)}$ from Example \ref{onenotation}.
	
	\begin{theorem}
		A map $f \colon X \rightarrow Y$ of $\Theta_n$-spaces is a Dwyer-Kan equivalence if and only if it is a levelwise weak equivalence.
	\end{theorem}
	
	\begin{proof}
		First, observe that any levelwise weak equivalence is necessarily a Dwyer-Kan equivalence, so we need only prove the converse statement.
		
		Suppose that $f \colon X \rightarrow Y$ is a Dwyer-Kan equivalence between $\Theta_n$-spaces.  Then for any $x,y \in X[0]_0$, we have that
		\[ M^\Theta_X(x,y) \rightarrow M^\Theta_Y(fx,fy) \]
		is a weak equivalence in $\Theta_{n-1}\css$, and that the map $X[0] \rightarrow Y[0]$ is an isomorphism on components.
		
		Recall that $M^\Theta_X(x,y)$ is a functor $\Theta_{n-1}^{\op} \rightarrow \SSets$, in fact a $\Theta_{n-1}$-space when $X$ is a $\Theta_n$-space, defined objectwise via pullbacks
		\[ \xymatrix{M^\Theta_X(x,y)(c) \ar[r] \ar[d] & X[1]([c]) \ar[d]\\
			\{(x,y)\} \ar[r] & X[0] \times X[0]. } \]
		In particular, we can understand $M^\Theta_X(x,y)$ by evaluating at the objects $[1]^{(i)}$ for $0 \leq i \leq n$, using the Segal conditions. Furthermore, the completeness conditions give us weak equivalences
		\[ X[1]^{(n-1)} \simeq X[1]^{(n)}_{\heq} \]
		for $0<i\leq n$.  Thus, it suffices to prove that the maps $X[1]^{(n)} \rightarrow Y[1]^{(n)}$ and $X[1]^{(n-1)} \rightarrow Y[1]^{(n-1)}$ are weak equivalences of simplicial sets.
		
		Restricting to homotopy equivalences, we take the pullback
		\[ \xymatrix{hM^\Theta_X(x,y)[1]^{(n-1)} \ar[r] \ar[d] & X[1]^{(n)}_{\heq} \ar[d] \\
			\{(x,y)\} \ar[r] & X[0] \times X[0],} \]
		which is a homotopy pullback since the right-hand map is a fibration, following from the Reedy fibrancy of $X$. Since weak equivalences are preserved by passing to subspaces of homotopy equivalences, by our assumption we know that
		\[ hM^\Theta_X(x,y)[1]^{(n-1)} \simeq hM^\Theta_Y(fx,fy)[1]^{(n-1)}. \]
		
		If we precompose with the weak equivalence $X[1]^{(n-1)} \rightarrow X[1]^{(n)}_{\heq}$ and the corresponding map for $Y$, we get a commutative square
		\[ \xymatrix{X[1]^{(n-1)} \ar[r] \ar[d] & X[0] \times X[0] \ar[d] \\
			Y[1]^{(n-1)} \ar[r] & Y[0] \times Y[0].} \]
		Since the fibers of the horizontal maps are weakly equivalent, we can conclude that this diagram is a homotopy pullback square.  By our assumption that $X[0] \rightarrow Y[0]$ induces an isomorphism on components, it follows that the map $X[1]^{(n-1)} \rightarrow Y[1]^{(n-1)}$ is a weak equivalence.
		
		Finally, we consider the diagram
		\[ \xymatrix{X[1]^{(n)} \ar[r] \ar[d] & X[0] \times X[0] \ar[d] \\
			Y[1]^{(n)} \ar[r] & Y[0] \times Y[0], } \]
		which is a pullback with horizontal maps fibrations, since $X$ and $Y$ are assumed to be Reedy fibrant.  Therefore, we obtain that the map $X[1]^{(n)} \rightarrow Y[1]^{(n)}$ is a weak equivalence.
	\end{proof}
	
	\section{Results on diagram categories with discreteness assumptions} \label{discreteresults}
	
	Since our goal is to develop models for $(\infty,n)$-categories using $\Theta_n^{\op}$-diagrams with discreteness conditions, generalizing the Segal category model, we need to prove analogues of several results that were used to establish a model structure for Segal categories.  Especially because we want to develop several different variants, in which we require discreteness at some levels but completeness at others, it is convenient to prove more general results.  In this section we consider functors $\mathcal C \rightarrow \SSets$ for some small category $\mathcal C$, and require the images of some specified objects of $\mathcal C$ to be discrete. For specificity, we give examples throughout this section of how this theory can be applied to the case when $\mathcal C = \Theta_2^{\op}$.
	
	We begin with the following result that can be obtained via the existence of a left Kan extension.
	
	\begin{prop} \label{discleftadjoint}
		Let $\mathcal C$ be a small category and $T \subseteq \ob(\mathcal C)$.  Consider the full subcategory of $\SSets^\mathcal C$ consisting all functors $X \colon \mathcal C \rightarrow \SSets$ such that $X(c)$ is discrete for all $c \in T$.  The inclusion of this subcategory into $\SSets^\mathcal C$ admits a left adjoint that we denote by $(-)_T$.  
	\end{prop}
	
	\begin{definition} \label{discretization}
		Let $\mathcal C$ be a small category and $T \subseteq \ob(\mathcal C)$.  Let $X \colon \mathcal C \rightarrow \SSets$.  The $T$-\emph{discretization} of $X$ is the functor $X_T \colon \mathcal C \rightarrow \SSets$ given by the image of $X$ under the left adjoint to the inclusion functor from Proposition \ref{discleftadjoint}.
	\end{definition}
	
	Roughly speaking, the passage from $X$ to $X_T$ takes each space $X(c)$ to the discrete space $\pi_0X(c)$ for each $c \in T$.  For $c \notin T$, the spaces $X(c)$ may or may not be affected, depending on how they interact with the spaces that are discretized.
	
	\begin{example}
		In the case of Segal categories, we considered functors $X \colon \Deltaop \rightarrow \SSets$ for which $X_0$ is discrete.  For $\mathcal C= \Deltaop$ and $T=\{[0]\}$, the $T$-discretization is the functor which is denoted by $(-)_r$ in \cite[\S 4]{thesis}.  In particular, in this case $(X_T)_0 = \pi_0 X_0$ and for $n \geq 1$, $(X_T)_n$ consists of the quotient space of $X_n$ given by collapsing the subspace in the image of iterated degeneracy maps from $X_0$ to its components.
	\end{example}
	
	This discretization functor is used in \cite[\S 5]{thesis} to modify a set of generating cofibrations in the Reedy model structure to serve as a generating set for the model structure for Segal categories; we apply a similar procedure for our model structures in Section \ref{modelstructures}.
	
	A different method of discretization is used in \cite[\S 6]{thesis} to describe the adjunction between the complete Segal space and Segal category model structures, and we generalize this approach in Section \ref{comparison}.  In preparation, we want to develop a general theory of skeleta and coskeleta to be used there.
	
	More specifically, given a subset $T$ of the objects of a small category $\mathcal C$, the skeleton and coskeleton of a functor $X \colon \mathcal C \rightarrow \SSets$ are given by the left and right adjoint, respectively, of a truncation functor that restricts to diagrams on the full subcategory of $\mathcal C$ with objects in $T$.  To that end, we make the following preliminary definition.
	
	\begin{definition}
		Let $\mathcal C$ be a small category and $T$ a subset of $\ob(\mathcal C)$.  Let $\mathcal C_T$ denote the full subcategory of $\mathcal C$ whose objects are those in $T$. The $T$-\emph{truncation} $\tr_T \colon \SSets^\mathcal C \rightarrow \SSets^{\mathcal C_T}$ is the functor induced by precomposition with the inclusion $\mathcal C_T \rightarrow \mathcal C$.
	\end{definition}
	
	The category $\SSets$ is sufficiently well-behaved that we can invoke the theory of Kan extensions to get the following result.
	
	\begin{prop}
		The functor $\tr_T$ admits a left adjoint $s_T$ and a right adjoint $c_T$.
	\end{prop}
	
	The following definition allows us to think of these adjoints as functors $\SSets^\mathcal C \rightarrow \SSets^\mathcal C$.
	
	\begin{definition}
		Let $\mathcal C$ be a small category, $X \colon \mathcal C \rightarrow \SSets$ a functor, and $T \subseteq \ob(\mathcal C)$.  The $T$-\emph{skeleton} of $X$ is $\sk_T(X):= s_T \circ \tr_T(X)$ and the $T$-\emph{coskeleton} of $X$ is $\cosk_T(X)=c_T \circ \tr_T(X)$.
	\end{definition}
	
	\begin{prop}
		The $T$-skeleton and $T$-coskeleton functors define an adjoint pair
		\[ \sk_T \colon \SSets^\mathcal C \rightleftarrows \SSets^\mathcal C \colon \cosk_T. \]
	\end{prop}
	
	\begin{proof}
		Using the above adjunction, for any $X, Y \colon \mathcal C \rightarrow \SSets$, we obtain natural isomorphisms
		\[ \begin{aligned}
			\Hom_{\SSets^\mathcal C}(\sk_T(X), Y) & \cong \Hom_{\SSets^\mathcal C}(s_T \circ \tr_T(X), Y) \\
			& \cong \Hom_{\SSets^{\mathcal C_T}}(\tr_T(X), \tr_T(Y)) \\
			& \cong \Hom_{\SSets^\mathcal C}(X, c_T \circ \tr_T(Y)) \\
			& \cong \Hom_{\SSets^\mathcal C}(X, \cosk_T(Y).
		\end{aligned} \]
	\end{proof}
	
	\begin{example}
		When $\mathcal C = \Deltaop$ and $T=\{[k] \mid 0 \leq k \leq m\}$ for some $m \geq 0$, then we recover the usual $m$-skeleton $\sk_m(X)$ and $m$-coskeleton $\cosk_m(X)$ of a simplicial space.  When $m=0$, the 0-skeleton of $X$ is the constant simplicial space on the simplicial set $X_0$, whereas the 0-coskeleton is given by $\cosk_0(X)_k = (X_0)^{k+1}$ for each $k \geq 0$.
	\end{example}
	
	\begin{example}
		Suppose that $\mathcal C = \Theta_2^{\op}$, and let us consider the coskeleta associated to subsets $T$ of 
		\[ S=\{[0], [1]([0]) \} \subseteq \ob(\Theta_2^{\op}). \]
		
		We start with the case when $T$ is the subset consisting of the object $[0]$; we denote the associated coskeleton functor by $\cosk_{[0]}$.  Given a functor $X \colon \Theta_2^{\op} \rightarrow \SSets$, we can use the fact that $\Theta_2$ is built from $\Delta$ in particular ways to describe $\cosk_{[0]}(X)$.  
		
		First, when we evaluate at any object of the form $[q]([0], \ldots, [0])$, we can use the description of the 0-coskeleton of a simplicial space to see that 
		\[ (\cosk_{[0]}X)[q]([0], \ldots, [0]) \cong X[0]^{q+1}. \]
		In particular, we have
		\[ (\cosk_{[0]}X)[1]([0]) \cong X[0]^{2}. \]
		Now, we can make use of the simplicial structure built into the objects $[1]([c])$ to observe that
		\[ (\cosk_{[0]}X)[1]([c]) \cong X[0]^{2}, \]
		and indeed one can check that
		\[ (\cosk_{[0]}X)[q]([c_1], \ldots, [c_q]) \cong \left(X[0]^{q+1}\right) \]
		for any $q$.
		
		Now, let us consider instead the case when $T$ is the subset containing only the object $[1]([0])$.  In this situation, the simplicial 0-coskeleton appears in the objects $[1]([c])$, for any $c \geq 0$, in that 
		\[ (\cosk_{[1]([0])}X)[1]([c]) \cong X[1]([0])^{c+1}. \]
		At the object $[0]$, we must have
		\[  (\cosk_{[1]([0])}X)[0] \cong \Delta[0]. \]
		For the object $[1]([1])$, we must get
		\[ (\cosk_{[1]([0])}X)[1]([1]) \cong X[1]([0]) \times X[1][0]; \]
		a general formula for evaluating at objects $[q]([c_1], \ldots, [c_q])$ quickly becomes more complicated.
		
		Finally, we consider the coskeleton associated to $S$ itself.  Here, we get 
		\[ (\cosk_S X)[0] = X[0] \]
		and
		\[ (\cosk_S X)[1]([0]) = X[1]([0]). \]
		It is not hard to check that 
		\[ (\cosk_S X)[q]([0], \ldots, [0]) \cong X[1]([0]) \times_{X[0]} \cdots \times_{X[0]} X[1]([0]).  \]
		We leave the descriptions upon evaluating at a general $[q]([c_1], \ldots, [c_q])$ to the reader.
	\end{example}
	
	\section{Model structures for $\Theta_n$-models with discreteness assumptions} \label{modelstructures}
	
	Now we turn to our question of having models for $(\infty, n)$-categories given by functors $\Thetanop \rightarrow \SSets$ that satisfy some discreteness conditions.  The first such condition we could ask for is on the level of objects, namely that a functor $X \colon \Thetanop \rightarrow \SSets$ have $X[0]$ a discrete simplicial set.  In other words, we want the simplicial set $X[1]([0])= X[1]^{(0)}$ to be discrete, but we also want to ask the same of other $X[1]^{(i)}$ for any $0 \leq i <n$.  
	
	Let us apply the results of Section \ref{discreteresults} to the category $\mathcal C=\Thetanop$ and the set 
	\[ S= \{[1]^{(i)} \mid 0 \leq i <n \}. \]
	
	\begin{definition}
		A $\Theta_n$-\emph{Segal precategory} is a functor $X \colon \Thetanop \rightarrow \SSets$ such that $X[1]^{(i)}$ is discrete for all $[1]^{(i)}$ in $S$.  It is a $\Theta_n$-\emph{Segal category} if it is additionally a $\Theta_n$-Segal space.
	\end{definition}
	
	\begin{remark}
		Observe that this definition includes an assumption that a $\Theta_n$-Segal precategory is Reedy fibrant.  As discussed in Remark \ref{reedyrmk}, this choice is less common for Segal categories and their generalizations, but it is convenient for us here, given what we want to prove. 
		
		As for Segal categories, the model structures that we develop here, with fibrant objects the Reedy fibrant $\Theta_n$-Segal precategories, have counterparts whose fibrant objects are projective fibrant instead.  We have chosen not to elaborate on this point in this paper, however.
	\end{remark}
	
	However, we can also restrict ourselves only to the objects in some subset $T \subseteq S$; indeed, we give our proofs in this generality.  In what follows, we assume that the set $S$ is fixed to be as defined above, but that $T$ is an arbitrary nonempty subset of $S$; we impose further conditions on $T$ momentarily.
	
	\begin{definition}
		Let $T \subseteq S$.  A $\Theta_n$-$T$-\emph{Segal precategory} is a functor $X \colon \Thetanop \rightarrow \SSets$ such that $X[1]^{(i)}$ is discrete for all $[1]^{(i)}$ in $T$.  It is a $\Theta_n$-$T$-\emph{Segal category} if it is additionally a $\Theta_n$-Segal space that is complete for objects $[1]^{(i)}$ in $S \setminus T$, in the sense that the map $X[1]^{(i)} \simeq X[1]^{(i+1)}_{\heq}$ is a weak equivalence.
	\end{definition}
	
	We denote the category of $\Theta_n$-$T$-Segal precategories by $\SSets^{\Thetanop}_T$.
	
	\begin{remark}  \label{bottomup}
		In principle, this definition is sensible for any subset $T \subseteq S$.  However, we claim that for many choices of $T$ we recover the same objects.  Suppose that $[1]^{(i)}$ is an object of $T$.  Then if $X$ is a $\Theta_n$-$T$-Segal category, it follows from the definition that $X[1]^{(i)}$ is a discrete simplicial set, and hence $X[1]^{(i)}_{\heq}$ must also be discrete.  Since we also assume that $X[1]^{(i-1)} \simeq X[1]^{(i)}_{\heq}$, it follows that $X[1]^{(i-1)}$ is homotopy discrete.  Indeed, since this weak equivalence is given by a degeneracy map, $X[1]^{(i-1)}$ is a retract of the discrete space $X[1]^{(i)}_{\heq}$, hence also discrete.
		
		Therefore, it suffices to consider $\Theta_n$-$T$-Segal precategories for which discretization is made ``from the bottom up".  Consequently, in what follows, we assume that $T=\{[1]^{(0)}, \ldots, [1]^{(j)} \}$ for some $0 \leq j \leq n-1$.
	\end{remark}
	
	A key feature that we want for our model structure is that the weak equivalences be the Dwyer-Kan equivalences.  To make sense of such maps for arbitrary diagrams, we need an appropriate localization functor taking a diagram to one that is a $\Theta_n$-$T$-Segal category.  Thus, we need to verify that such a localization results in a $\Theta_n$-$T$-Segal precategory rather than a more general diagram $\Thetanop \rightarrow \SSets$, generalizing the argument for Segal categories in \cite[\S 5]{thesis}.
	
	\begin{prop} \label{discreteloc}
		If $X$ is a $\Theta_n$-$T$-Segal precategory, then there exists a functor 
		\[ L \colon \SSets^{\Thetanop}_{T} \rightarrow \SSets^{\Thetanop}_{T} \] 
		such that $LX$ is a $\Theta_n$-$T$ Segal category that is weakly equivalent to $X$ in $\Theta_n\sesp$.
	\end{prop}
	
	\begin{proof}
		Let us first suppose that $T=S$, so that we discretize at every level.  We want to modify the localization functor $L$ in $\Theta_n\sesp$ in such a way that if $X[1]^{(i)}$ is discrete for some $i$, then so is $(LX)[1]^{(i)}$.
		
		We can think of this original localization as occurring in two stages: first, it provides a Reedy fibrant replacement of an object $X$, and then it gives an additional localization so that the resulting object satisfies the Segal conditions. Let us first look at the Reedy fibrant replacement process.
		
		Recall from Section \ref{thetanreview} that the generating acyclic cofibrations for the Reedy structure can be taken to be those of the form
		\[ \partial \Theta[q](c_1, \ldots, c_q) \times \Delta[m] \cup \Theta[q](c_1, \ldots, c_q) \times \Lambda^k[m] \rightarrow \Theta[q](c_1, \ldots, c_q) \times \Delta[m] \]
		where $m \geq 1$, $0 \leq k \leq m$, and $[q](c_1, \ldots, c_q) \in \ob(\Theta_n)$.  Thus, a functorial Reedy fibrant replacement in $\SSets^{\Thetanop}$ is given by taking iterated pushouts along these maps.  However, when $[q](c_1, \ldots, c_q) = [1]^{(i)}$ for some $0 \leq i < n$, the resulting pushouts may not satisfy the required discreteness condition on $X[1]^{(i)}$.  
		
		For example, if $q=0$, taking a pushout of $X$ along such a map to get some $X'$ is effectively given by taking a pushout
		\[ \xymatrix{\Lambda^k[m] \ar[r] \ar[d] & X[0] \ar[d] \\
			\Delta[m] \ar[r] & X'[0]. } \]
		Such a pushout need not be discrete.  On the other hand, if $X[0]$ is discrete, then the image of $\Lambda^k[m]$ is one of the points of $X[0]$ and therefore this map extends to a map $\Delta[m] \rightarrow X[0]$.  In other words, $X$ already satisfies the desired lifting condition with respect to the generating acyclic cofibrations for which $q=0$, so there is no harm in omitting these maps when taking the localization.
		
		A similar argument works for any $[q](c_1, \ldots, c_q) = [1]^{(i)}$, so we obtain a Reedy fibrant replacement for $X$ by taking iterated pushouts along the maps
		\[ \xymatrix{\partial \Theta[q](c_1, \ldots, c_q) \times \Delta[m] \cup \Theta[q](c_1, \ldots, c_q) \times V[m,k] \ar[d] \\
			\Theta[q](c_1, \ldots, c_q) \times \Delta[m]} \]
		where $m \geq 1$, $0 \leq k \leq m$, and $[q](c_1, \ldots, c_q) \in \ob(\Theta_n)$ is not an object $[1]^{(i)}$ of $T$.
		
		Now, let us turn to the localization to obtain a Segal $\Theta_n$-space.  Let us first consider the maps
		\[ G[q](c_1, \ldots, c_q) \rightarrow \Theta[q](c_1, \ldots, c_q) \]
		for any $q \geq 0$ and $c_i \in \ob(\Theta_{n-1})$.  
		Since these maps are cofibrations between cofibrant objects in the Reedy model structure, and $X$ can now be assumed to be Reedy fibrant, we know that $X$ is local with respect to these maps precisely when a lift exists in any diagram
		\[ \xymatrix{\partial \Delta[m] \ar[r] \ar[d] & \Map(\Theta[q](c_1, \ldots, c_q), X) \ar[d] \\
			\Delta[m] \ar[r] \ar@{-->}[ur] & \Map(G[q](c_1, \ldots, c_q), X).} \]
		The existence of such a lift is equivalent to the existence of a lift in the diagram
		\[ \xymatrix{G[q](c_1, \ldots, c_q) \times \Delta[m] \cup \Theta[q](c_1, \ldots, c_q) \times \partial \Delta[m] \ar[r] \ar[d] & X \\
			\Theta[q](c_1, \ldots, c_q) \times \Delta[m] \ar@{-->}[ur] & .} \]
		Thus, it is these maps on the left-hand side that we take pushouts along to obtain a local object.  There is potential concern when $q=0$, but in that case these maps are the identity, since $G[0] =\Delta[0]$.  Therefore no problems arise when we take pushouts along such maps.  When $q=1$, taking pushouts along these maps again has no effect because $G[1](c) = \Theta[1](c)$ for any $c \in \ob(\Theta_{n-1})$.
		
		The other maps used in the localization can be shown similarly to present no difficulties, since they are inductively defined, essentially again using the fact that $G[1](c) = \Theta[1](c)$ for any $c$.  By way of illustration, if 
		\[ G[r](d_1, \ldots, d_r) \rightarrow \Theta[r](d_1, \ldots, d_r) \]
		are the analogous maps in $\SSets^{\Theta_{n-1}^{\op}}$, then we need to localize $\SSets^{\Thetanop}$ with respect to the maps
		\[ V[1](G[r](d_1, \ldots, d_r)) \rightarrow V[1](\Theta[r](d_1, \ldots, d_r)). \]
		These maps are the identity upon evaluation at $[0]$, so taking a pushout along them presents no problem at that level.  The argument that taking pushouts along these maps when $[r](d_1, \ldots, d_r)= [1]^{(1)}= [1]([0])$ does not alter discreteness is essentially the same as the argument above for $q=0$, and with a similar shift for the other $[1]^{(i)}$.
		
		Thus, we have that applying the described modification of Reedy fibrant replacement, followed by the usual $\Theta_n$-Segal localization, results in a $\Theta_n$-Segal category, as we wished to show.
		
		Finally, we consider the case when $T \neq S$, so that $T=\{[1]^{(i)} \mid 0 \leq i \leq j\}$ for some fixed $j<n$.  Now, we need to localize further so that at the levels corresponding to elements of $S \setminus T$, we have that $LX$ satisfies completeness.  Recalling the notation set up in the paragraph before Theorem \ref{thetanmc}, we define the set 
		\[ \mathcal T'_{n,j} = \Cpt_{\Theta_n} \cup V[1](\Cpt_{\Theta_{n-1}}) \cup \cdots \cup V[1]^{n-j-1}(\Cpt_{\Theta_{j+1}}). \]
		A similar argument as above shows that localizing with respect to these maps does not affect discreteness at the previously indicated levels, and results in an object with the required completeness conditions, namely, a $\Theta_n$-$T$ Segal category.
	\end{proof}
	
	We use this result to make sense of Dwyer-Kan equivalences $X \rightarrow Y$ between functors $\Thetanop \rightarrow \SSets$ that are required to be discrete at those objects $[1]^{(i)}$ in $T$ but that may not be Segal $\Theta_n$-spaces.
	
	\begin{definition} \label{dkequivdefn}
		A map $X \rightarrow Y$ in $\SSets^{\Thetanop}_{T}$ is a \emph{Dwyer-Kan equivalence} if the map $LX \rightarrow LY$ is fully faithful and essentially surjective, i.e., a Dwyer-Kan equivalence in the sense of Segal $\Theta_n$-spaces. 
	\end{definition}
	
	\begin{theorem} \label{discretthetamc}
		There is a model structure $\Theta_n\Secat_T$ on the category of $\Theta_n$-$T$-Segal precategories in which:
		\begin{enumerate}
			\item the weak equivalences are the Dwyer-Kan equivalences;
			
			\item the cofibrations are the monomorphisms; and 
			
			\item the fibrant objects are the $\Theta_n$-$T$-Segal categories that are complete at every element $[1]^{(i)}$ of $S \setminus T$. 
		\end{enumerate}
	\end{theorem}
	
	The last condition means that a fibrant object $X$ has $X[1]^{(i)}$ discrete when $[1]^{(i)}$ is an element of $T$, i.e., when $0 \leq i \leq j$, but for $j>i$ we have that $X[1]^{(i)}$ is weakly equivalent to the space of homotopy equivalences in $X[1]^{(i+1)}$, just as for $\Theta_n$-spaces.
	
	Our proof follows the general strategy used to establish the model structure for Segal categories in \cite[\S 5]{thesis}.  Some of the proofs there are fairly formal and can be applied nearly identically, so we leave modifying them to our context here as an exercise for the reader. We give proofs, however, for those results whose generalization is less clear, as well as some for which we have found more efficient proofs than the originals.  In \cite[\S 6]{inftyn1} we proved an analogous result using a different method, and one could take the same approach here.  We have chosen former method here because it gives more explicit descriptions of some of the maps used.
	
	The first step in proving this theorem is to find candidates for the generating cofibrations and generating acyclic cofibrations.  We apply the techniques of Section \ref{discreteresults} to modify the Reedy generating cofibrations so that they satisfy the necessary discreteness assumptions.  Using the notation there, we let $\mathcal C = \Thetanop$ and 
	\[ T = \left\{ [1]^{(i)} \mid 0 \leq i \leq j \right\} \]
	for some fixed $j<n$.
	
	We define a set $I^{n,T}$ of proposed generating cofibrations, consisting of the maps
	\[ \xymatrix{(\partial \Delta[m] \times \Theta[q](c_1, \ldots, c_q) \cup \Delta[m] \times \partial \Theta[q](c_1, \ldots, c_q))_T \ar[d] \\
		(\Delta[m] \times \Theta[q](c_1, \ldots, c_q))_T} \]
	for all $m, q \geq 0$ and $(c_1, \ldots, c_q) \in \ob(\Theta_{n-1})^{q+1}$, where $(-)_T$ is the discretization functor from Definition \ref{discretization}.
	
	We do not expect such a nice description for the generating acyclic cofibrations, so in line with the approach we used for Segal categories we take $J^{n,T}$ to be the set of isomorphism classes of maps $A \rightarrow B$ such that
	\begin{enumerate}
		\item the map $A \rightarrow B$ is a monomorphism and a Dwyer-Kan equivalence, and 
		
		\item for all $m, q \geq 0$ and $(c_1, \ldots, c_q) \in \ob(\Theta_{n-1})$, the simplicial sets $A_{[q](c_1, \ldots, c_q),m}$ and $B_{[q](c_1, \ldots, c_q),m}$ have only countably many simplices.
	\end{enumerate}
	Observe that $J^{n,T}$ does not at first glance seem to depend on the set $T$, but the maps in this set are between objects that satisfy the required discreteness conditions at the objects of $T$.  Thus, these sets are in fact different for varying $T$.
	
	The proof of the following result involves some additional techniques, so we defer it to Section \ref{deferred}.
	
	\begin{prop} \label{thesis5.8} \label{thesis5.10}
		Maps with the right lifting property with respect to $I^{n,T}$ are precisely the maps that are both fibrations and weak equivalences.
	\end{prop}
	
	We turn to some properties of the set $J^{n,T}$.
	
	\begin{prop} \label{acycliccofs}
		\begin{enumerate}
			\item \label{thesis5.7}
			Any map that is both a cofibration and a weak equivalence can be written as a directed colimit of pushouts along maps in $J^{n,T}$.
			
			\item \label{thesis5.11}
			Pushouts along maps in $J^{n,T}$ are cofibrations and weak equivalences.
			
			\item \label{thesis5.12}
			Any $J^{n,T}$-cofibration is an $I^{n,T}$-cofibration and a weak equivalence.
		\end{enumerate}
	\end{prop}
	
	\begin{proof}
		The proof of \eqref{thesis5.7} is technical but follows the same line of argument as \cite[5.7]{thesis}, so we do not repeat it here.
		
		To prove \eqref{thesis5.11}, first notice that any $j \colon A \rightarrow B$ in $J^{n,T}$ is an acyclic cofibration in $\Thetansp$.  Since pushouts along maps in $\SSets^{\Thetanop}_{T}$ preserve the required discreteness conditions, the resulting map is still in $\SSets^{\Thetanop}_{T}$.  Furthermore, it is an acyclic cofibration in $\Thetansp$, so a monomorphism which is a Dwyer-Kan equivalence, as we wished to show.
		
		Finally, we prove \eqref{thesis5.12}. By definition and \eqref{thesis5.7}, a $J^{n,T}$-cofibration is a map with the left lifting property with respect to the fibrations. Similarly, using Proposition \ref{thesis5.8}, an $I^{n,T}$-cofibration is a map with the left lifting property with respect to the fibrations which are Dwyer-Kan equivalences.  Any map with the left lifting property with respect to the fibrations also has the left lifting property with respect to the fibrations that are also weak equivalences, so we only need to check that such a map is a weak equivalence.
		
		Let $f \colon A \rightarrow B$ be such a map.  By \eqref{thesis5.11}, a pushout along maps of $J^{n,T}$ is an acyclic cofibration.  Therefore, we can use the small object argument to factor $f \colon A \rightarrow B$ as $A \overset{\simeq}{\hookrightarrow} A' \twoheadrightarrow B$, so there exists a lift in the diagram
		\[ \xymatrix{A \ar[r]^{\simeq} \ar[d] & A' \ar[d] \\
			B \ar[r]^{\id} \ar@{-->}[ur] & B.} \]
		Therefore, $A \rightarrow B$ is a retract of $A \rightarrow A'$ and hence a weak equivalence.
	\end{proof}
	
	We now have all the ingredients we need to establish the model structure.
	
	\begin{proof}[Proof of Theorem \ref{discretthetamc}]
		We apply the conditions of Theorem \ref{cofgen}.  The category of $\Theta_n$-$T$-Segal precategories has all small limits and colimits, since they are taken levelwise and therefore do not affect the discreteness assumptions.  Similarly, Dwyer-Kan equivalences satisfy the two-out-of-three property and are closed under retracts.
		
		The set $I^{n,T}$ permits the small object argument because the generating cofibrations in the Reedy model structure do, and applying the discretization functor does not affect this property.  The objects $A$ that appear as the sources of the maps in $J^{n,T}$ are small using the same argument as for Segal categories \cite[5.1]{thesis}, so the set $J^{n,T}$ permits the small object argument. Thus, condition \eqref{cofgen1} is satisfied.
		
		Condition \eqref{cofgen2} is precisely the statement of Proposition \ref{acycliccofs}\eqref{thesis5.12}. Condition \eqref{cofgen3} and condition \eqref{cofgen4}(b) are together precisely the statement of Proposition \ref{thesis5.8}.
	\end{proof}	
	
	Finally, we conclude with establishing some additional structure that this model structure possesses.
	
	\begin{prop}
		The model structure $\Theta_n\Secat_T$ is simplicial and cartesian.
	\end{prop}
	
	\begin{proof}
		We give the proof that the model structure is cartesian; the proof that it is simplicial can be proved similiarly, or using an argument like the one used for Segal categories in \cite[6.3]{ginfty1}.
		
		We know that every object in $\Theta_n\Secat_T$, in particular the terminal object, is cofibrant, so it remains to check the other two conditions of Definition \ref{cartdef}.
		
		Consider the category $\SSets^{\Thetanop}$ with the model structure whose fibrant objects satisfy the Segal conditions at all levels and completeness conditions at each object $[1]^{(i)}$ in $S \setminus T$.  This model category, that we denote by $\mathcal M$ here for simplicity, is cartesian, using \cite[6.1, 8.1]{rezktheta}.  Note that every weak equivalence in $\Theta_n\Secat_T$ is also a weak equivalence in $\mathcal M$, and similarly for cofibrations, so we can use this structure on $\mathcal M$ to establish the conditions we need.
		
		To show that the underlying category $\SSets^{\Thetanop}_T$ is cartesian closed, note that if $X$ and $Y$ are discrete at some level $[1]^{(i)}$, then so is their product $X \times Y$.  We claim the same is true of the mapping object $Y^X$, which is defined by
		\[ (Y^X)_{[m](c_1, \ldots, c_m)} = \Map(X \times \Theta[m](c_1, \ldots, c_m), Y). \]
		A straightforward computation shows that if $X_{[1]^{(i)}}$ and $Y_{[1]^{(i)}}$ are both discrete, then so is $(Y^X)_{[1]^{(i)}}$.  The required compatibility between the cartesian product and mapping object follows because it holds in $\mathcal M$.
		
		Similarly, suppose that $f \colon X \rightarrow X'$ and $g \colon Y \rightarrow Y'$ are cofibrations in $\Theta_n\Secat_T$, and in particular discrete at every $[1]^{(i)}$ in $T$.  Then the pushout-corner map is again a monomorphism that is discrete at the same levels.  Using left properness, which follows since all objects are cofibrant, and the two-out-of-three property, one can check that this map is a weak equivalence if either $f$ or $g$ is.
	\end{proof}
	
	\begin{remark}
		When $n=1$, the previous proposition recovers the result of Simpson that the model structure for Segal categories is cartesian \cite[19.3.3]{simpson}.
	\end{remark}
	
	\section{Comparison of models} \label{comparison}
	
	In this section, we establish Quillen equivalences between the different model structures from the previous section, for varying $T$, and with $\Thetansp$.  The strategy of proof is very similar to the comparison between the model structures for Segal categories and complete Segal spaces in \cite[\S 6]{thesis}.
	
	
	In this section, let $T_j = \{[1]^{(i)} \mid 0 \leq i \leq j \}$ for a fixed $j <n$.  We want to prove that the inclusion $I$ of the category of $\Theta_n$-$T_j$-Segal precategories into the category of $\Theta_n$-$T_{j-1}$-Segal precategories has a right adjoint, and further that this adjoint pair induces a Quillen equivalence between the model structures for $\Theta_n$-$T_j$-Segal categories and $\Theta_n$-$T_{j-1}$-Segal categories.  In other words, for that value of $j$, we drop the assumption that $X([1]^{(j)})$ be discrete, but ask instead for the corresponding completeness condition. If $j=0$, then we take $T_{-1} = \varnothing$, in which case we get the comparison with $\Theta_n$-spaces, for which no spaces in the diagram are required to be discrete.  In light of Remark \ref{bottomup}, we thus obtain all the comparisons we are interested in.
	
	Let $W$ be an object of $\SSets^{\Thetanop}_{T_{j-1}}$.  Consider the objects $U=\cosk_{[1]^{(j)}}(W)$ and $V= \cosk^0_{[1]^{(j)}}(W)$, where the latter is defined to be the coskeleton of the discrete functor $\Thetanop \rightarrow \SSets$ given by $[q](c_1, \ldots, c_q) \mapsto W[q](c_1, \ldots, c_q)_0$. Define $RW$ to be the pullback in the diagram
	\[ \xymatrix{RW \ar[r] \ar[d] & V \ar[d] \\
		W \ar[r] & U.} \]
	Note that $RW$ is a $\Theta_n$-$T_j$-Segal precategory, since the effect of this process is to discretize $W$ at the object $[1]^{(j)}$.
	Observe that this construction defines a functor
	\[ R \colon \SSets^{\Thetanop}_{T_{j-1}} \rightarrow \SSets^{\Thetanop}_{T_j}. \]
	
	Now, we want to prove the higher analogue of Theorem \ref{secatcss}, that this functor $R$ is the right adjoint of a Quillen equivalence.  Verifying the following result is a straightforward generalization of the argument used to prove \cite[6.1]{thesis}.
	
	\begin{prop}
		The functor $R \colon \SSets^{\Thetanop}_{T_{j-1}} \rightarrow \SSets^{\Thetanop}_{T_j}$ is right adjoint to the inclusion $I$.
	\end{prop}
	
	Now, we need to show that this adjoint pair respects the model structures of interest.
	
	\begin{theorem}
		The adjoint pair 
		\[ \xymatrix@1{I \colon \Theta_n\Secat_{T_j} \ar@<.5ex>[r] & \Theta_n\Secat_{T_{j-1}} \colon R \ar@<.5ex>[l]} \] 
		is a Quillen equivalence. 
	\end{theorem}
	
	\begin{proof}
		To show that $(I,R)$ is a Quillen pair, we to show that the inclusion map $I$ preserves cofibrations and acyclic cofibrations.  It preserves cofibrations because they are precisely the monomorphisms in each model structure; it preserves all weak equivalences, and in particular the acyclic cofibrations, since a map is a weak equivalence in either model structure if and only if it is a Dwyer-Kan equivalence.  
		
		To show that this Quillen pair is a Quillen equivalence, we need to show that $I$ reflects weak equivalences between cofibrant objects and that for any $\Theta_n$-$T_{j-1}$-Segal category $W$, the map $I((RW)^c) = IRW \rightarrow W$ is a weak equivalence in $\Theta_n\Secat_{T_j}$.   The fact that $I$ reflects weak equivalences between cofibrant objects follows from the fact that the weak equivalences in each model structure are precisely the Dwyer-Kan equivalences.  
		
		It remains to show that the map $RW \rightarrow W$ in the pullback diagram
		\[ \xymatrix{RW \ar[r] \ar[d]_j & V \ar[d] \\
			W \ar[r] & U} \] 
		is a Dwyer-Kan equivalence.   By the definition of $RW$, we have $(RW)[0]_0=W[0]_0$, so induced map $\Ho^\Theta(RW) \rightarrow \Ho^\Theta(W)$ is a bijection on objects, hence essentially surjective.
		
		Finally, we need to show that, for any $x,y \in W[0]_0 = (RW)[0]_0$, the map
		\[ M^\Theta_{RW}(x,y) \rightarrow M^\Theta_W(x,y) \]
		is a weak equivalence in $\Theta_{n-1}\css$.  We claim that it is in fact a levelwise weak equivalence of functors $\Theta_{n-1}^{\op} \rightarrow \SSets$.  Since $W$ is assumed to be fibrant, it satisfies the Segal conditions; it follows from the pullback defining it that $RW$ does as well.  Thus, it suffices to verify that the map
		\[ M_{RW}^\Theta(x,y)[1]^{(i)} \rightarrow M_W^\Theta(x,y)[1]^{(i)} \]
		is a weak equivalence of simplicial sets for any $0 \leq i <n-1$.  We hence consider the diagram of homotopy fibers
		\begin{equation} \label{fibers}
			\xymatrix{M_{RW}^\Theta(x,y)[1]^{(i)} \ar[r] \ar[d] & (RW)[1]^{(i)} \ar[r] \ar[d] & (RW)[0] \times (RW)[0] \ar[d]^= \\
				M_W^\Theta(x,y)[1]^{(i)} \ar[r] & W[1]^{(i)} \ar[r] & W[0] \times W[0].}
		\end{equation} 
		
		First, let us consider $T_0=\{[0]\}$ and $T_{-1}=\varnothing$.  Then $(RW)[1]^{(i)}$ and $W[1]^{(i)}$ differ only in that the former contains degeneracies of the higher-dimensional simplices of $W[0]$.  Since we are taking the homotopy fiber over a 0-simplex of $W[0] \times W[0]$, however, these degeneracies do not appear in that homotopy fiber.  It follows that the middle vertical map of \eqref{fibers} is a weak equivalence, hence the left-hand vertical map is also.
		
		Now consider $T_j$ for $j \geq 1$.  Then $W[1]^{(i)}$ is already discrete for each $0 \leq i <j$, and $(RW)[1]^{(i)} =W[1]^{(i)}$ for these values of $i$.  When $i=j$, the middle vertical map of \eqref{fibers} is given by the inclusion of the discrete subspace $W[1]^{(j)}_0 \rightarrow W[1]^{(j)}$, and an argument similar to the one for $T_0$ shows that the induced map on homotopy fibers is a weak equivalence.  Finally, when $i>j$, the middle vertical map of \eqref{fibers} is given by the inclusion of a subspace that does not include higher degenerate elements coming from $W[1]^{(j)}$.  It follows that this map is a weak equivalence, hence the left-hand vertical map in \eqref{fibers} is also, as we needed to show.  
	\end{proof}
	
	\section{Proof of Proposition \ref{thesis5.8}} \label{deferred}
	
	In this section, we complete the proof of Proposition \ref{thesis5.8}, which tells us that maps with the right lifting property with respect to $I^{n,T}$ are precisely the maps that are both fibrations and Dwyer-Kan equivalences.  Because the two implications are proved quite differently, for clarity we separate them into the following two propositions.
	
	\begin{prop} \label{forward}
		If $f \colon X \rightarrow Y$ is a map of $\Theta_n$-$T$-Segal precategories with the right lifting property with respect to the maps in $I^{n,T}$, then it is a fibration and a Dwyer-Kan equivalence.
	\end{prop}
	
	\begin{prop} \label{backward}
		If $f \colon X \rightarrow Y$ is a map of $\Theta_n$-$T$-Segal precategories that is both a fibration and a Dwyer-Kan equivalence, then it has the right lifting property with respect to the maps in $I^{n,T}$.
	\end{prop}
	
	The more difficult of the two is Proposition \ref{forward}, which requires several preparatory lemmas.  For the beginning steps we work incrementally, starting with small values of $n$, to build intuition for the complicated notation we must inevitably use.
	
	Before delving into the details, recall that the definition of Dwyer-Kan equivalence is given in terms of mapping objects in a $\Theta_n$-space.  For the arguments we make in this section, we want to reformulate this definition somewhat.  Given $X \colon \Thetanop \rightarrow \SSets$, $c \in \ob(\Theta_{n-1})$, and $(v_0, v_1) \in X[0]_0 \times X[0]_0$, let $X[1](c)(v_0, v_1)$ be the fiber of the map $X[1](c) \rightarrow X[0] \times X[0]$ over $(v_0, v_1)$.  Then it is straightforward to check that
	\[ X[1](c)(v_0, v_1) = M_X^\Theta(v_0, v_1)(c). \]
	We use this notation, rather than the mapping object notation, for the remainder of this section.
	
	We want to understand the behavior of maps with the right lifting property with respect to maps in $I^{n,T}$.  However, it is easier to get a handle on maps with the right lifting property with respect to the maps in a related set that we denote by $I^{n,T}_f$, so we first consider these maps, which we develop in some detail.
	
	Given any object $[1]^{(i)}$ in $T$, we can evaluate any representable functor $\Theta[q](c_1, \ldots, c_q)$ at $[1]^{(i)}$ to obtain a set that we think of as a doubly constant functor $\Thetanop \times \Deltaop \rightarrow \Sets$, and we denote it by $\Theta[q](c_1, \ldots, c_q)_{[1]^{(i)}}$.   When $i=0$, we have $\Theta[q](c_1, \ldots, c_q)_{[0]} = \Hom([0], [q](c_1, \ldots, c_q))$, which is the set consisting of $q+1$ elements.  
	
	For any $m \geq 0$, object $[q](c_1, \ldots, c_q)$ of $\Theta_n$, and element $[1]^{(i)}$ in $T$, we have the projection and inclusion maps
	\[ \Theta[q](c_1, \ldots, c_q)_{[1]^{(i)}} \leftarrow \Delta[m] \times \Theta[q](c_1, \ldots, c_q)_{[1]^{(i)}} \rightarrow \Delta[m] \times \Theta[q](c_1, \ldots, c_q). \]
	Keeping $m$ and $[q](c_1, \ldots, c_q)$ fixed but varying over all $[1]^{(i)}$ in $T$, take the diagram given by all such maps and denote its colimit by $Q^{n,T}_{m, \cu}$.  Denote the colimit of the analogous diagram with $\Delta[m]$ replaced by $\partial \Delta[m]$ by $P^{n,T}_{m, \cu}$.  There are natural maps $P^{n,T}_{m, \cu} \rightarrow Q^{n,T}_{m, \cu}$, and it is this collection of maps we want to consider.  Specifically, define
	\[ I^{n,T}_{f} = \{P^{n,T}_{m, \cu} \rightarrow Q^{n,T}_{m, \cu} \mid m \geq 0, [q](c_1, \ldots, c_q) \in \ob(\Theta_n) \}. \]
	
	\begin{remark}
		This set of maps can be regarded as a set of generating cofibrations for a model structure for $\Theta_n$-$T$-Segal precategories that is more closely related to the projective model structure.  Indeed, these maps are designed to be an appropriate discretization of the generating cofibrations of the projective model structure on $\SSets^{\Thetanop}$.  The subscript $f$ is meant to be suggestive of this fact, even though we have removed the corresponding $c$ subscript from the corresponding injective or Reedy version.  We refer the reader to \cite[\S 6]{thesis} for a more detailed motivation for these kinds of maps in the case of Segal categories.
	\end{remark}
	
	Our first step is to obtain a good description of maps $P_{m,\cu}^{n,T} \rightarrow X$ and $Q_{m,\cu}^{n,T} \rightarrow X$ for general $X$.  Let us first recall the the case where $n=1$, namely Segal categories, which was treated in \cite[\S 4]{thesis}.  With a view toward generalization to higher $n$, and recalling that $\Theta_1=\Delta$, we denote the representable functor $\Deltaop \rightarrow \Sets$ on the object $[q]$ by $\Theta[q]$, rather than $\Delta[q]$. Here, we regard it as a discrete functor $\Deltaop \rightarrow \SSets$; we similarly treat the representable simplicial set $\Delta[m]$ as a constant functor $\Thetanop \rightarrow \SSets$.  In this case, there is only one object whose image we discretize, namely $[0]$, so there is no need to consider different choices of subsets $T$.  We thus simplify the notation for the moment and simply write $P_{m,q}$ and $Q_{m,q}$.  
	Let us recall some notation.  If $X$ is a Segal precategory and $(v_0, \ldots, v_q) \in X_0^{q+1}$, let $X_q(v_0, \ldots, v_q)$ denote the fiber of the natural map $X_q \rightarrow X_0^{q+1}$ given by iterated face maps.
	
	The following lemma was proved in \cite[\S 4]{thesis}; we sketch a proof here for the purposes of guiding our generalizations of it.
	
	\begin{lemma}
		When $n=1$, for fixed $m, q \geq 0$ and Segal precategory $X$, there are isomorphisms
		\[ \Hom(P_{m,q}, X) \cong \coprod_{(v_0, \ldots, v_q)} \Hom(\partial \Delta[m], X_q(v_0, \ldots, v_q)) \]
		and
		\[ \Hom(Q_{m,q}, X) \cong \coprod_{(v_0, \ldots, v_q)} \Hom(\Delta[m], X_q(v_0, \ldots, v_q)), \]
		where $(v_0, \ldots, v_q) \in X[0]_0^{q+1}$.
	\end{lemma}
	
	\begin{proof}[Sketch of proof]
		We summarize the argument for $Q_{m,q}$; the one for $P_{m,q}$ is similar.  We have defined $Q_{m,q}$ as the pushout in the diagram
		\[ \xymatrix{\Delta[m] \times \Theta[q]_0 \ar[r] \ar[d] & \Delta[m] \times \Theta[q]_0 \ar[d] \\
			\Theta[q]_0 \ar[r] & Q_{m,q}. } \]
		Applying the functor $\Hom(-,X)$, we obtain a pullback diagram of sets
		\[ \xymatrix{\Hom(Q_{m,q}, X) \ar[r] \ar[d] & \Hom(\Delta[m] \times \Theta[q], X) = X[q]_m \ar[d] \\
			X[0]_0^{q+1} = \Hom(\Theta[q]_0, X) \ar[r] & \Hom(\Delta[m] \times \Theta[q]_0, X) = X[0]_m^{q+1}.} \]
		But, this pullback can also be described as
		\[ \coprod_{(v_0, \ldots, v_q)} \Hom(\Delta[m], X_q(v_0, \ldots, v_q)). \]
	\end{proof}
	
	Now, we want to generalize this argument.  Not surprisingly, the combinatorics get quite complicated as we require that more levels of $X$ be discrete.  Let us start with $n=2$ and the set $S=\{[0], [1]([0])\}$ before attempting a full generalization to higher $n$ and proper subsets of $S$.  
	
	For a fixed $m \geq 0$ and $[q](c_1, \ldots c_q)$, we have defined $Q_{m, \cu}^{2,S}$ to be the colimit of the diagram
	\begin{equation} \label{qmqc}
		\xymatrix{& \Delta[m] \times \Theta[q](c_1, \ldots, c_q) & \\
			\Delta[m] \times \Theta[q](c_1, \ldots, c_q)_{[0]}  \ar[ur] \ar[d] && \Delta[m] \times \Theta[q](c_1, \ldots, c_q)_{[1]([0])}  \ar[ul] \ar[d] \\
			\Theta[q](c_1, \ldots, c_q)_{[0]} && \Theta[q](c_1, \ldots, c_q)_{[1]([0])}.}
	\end{equation} 
	If we focus on the two arrows on the left-hand side of this diagram, and take the pushout thereof, the situation is very similar to the one from the $n=1$ case.  Namely, if we apply the functor $\Hom(-,X)$ to these two arrows, we get a diagram
	\[ \xymatrix{& \Hom(\Delta[m] \times \Theta[q](c_1, \ldots, c_q), X) \ar[d] \\
		\Hom(\Theta[q](c_1, \ldots, c_q)_{[0]}, X) \ar[r] & \Hom(\Delta[m] \times \Theta[q](c_1, \ldots, c_q)_{[0]}, X) } \]
	that can be rewritten as
	\[ \xymatrix{ & X[q](c_1, \ldots, c_q)_m \ar[d] \\
		X[0]^{q+1}_0 \ar[r] & X[0]^{q+1}_m.} \]
	The pullback of this diagram is given by
	\[ \coprod_{(v_0, \ldots, v_q)} X[q](c_1, \ldots, c_q)(v_0, \ldots, v_q)_m, \]
	where similarly to above, $X[q](c_1, \ldots, c_q)(v_0, \ldots, v_q)$ is the fiber of the map $X[q](c_1, \ldots, c_q) \rightarrow X_0^{q+1}$.  This pullback is isomorphic to
	\[ \coprod_{(v_0, \ldots, v_q)} \Hom(\Delta[m], X[q](c_1, \ldots, c_q)(v_0, \ldots, v_q)). \]
	
	Now, let us treat the pushout of the two right-hand arrows of \eqref{qmqc} analogously.  We first apply the functor $\Hom(-,X)$ and then describe the pullback of the resulting diagram.  So, our first step is to describe the objects of 
	\[ \xymatrix{& \Hom(\Delta[m] \times \Theta[q](c_1, \ldots, c_q), X) \ar[d] \\
		\Hom(\Theta[q](c_1, \ldots, c_q)_{[1]([0])}, X) \ar[r] & \Hom(\Delta[m] \times \Theta[q](c_1, \ldots, c_q)_{[1]([0])}, X).} \]
	Looking at the bottom row, these sets should be described as some coproduct of copies of $X[1]([0])_0$ and $X[1]([0])_m$, respectively, indexed by maps $[1]([0]) \rightarrow [q](c_1, \ldots, c_q)$ in $\Theta_2$.  We denote by $\theta(1,\cu)$ the number of such maps.  Thus, we can describe this diagram instead as
	\[ \xymatrix{ & X[q](c_1, \ldots, c_q)_m \ar[d] \\
		X[1]([0])_0^{\theta(1,\cu)} \ar[r] & X[1]([0])_m^{\theta(1,\cu)}} \]
	whose pullback is
	\[ \coprod_{(w_1, \ldots, w_{\theta(1,\cu)})} X[q](c_1, \ldots, c_q)(w_0, \ldots, w_{\theta(1,\cu)})_m. \]
	This pullback can be written alternatively as
	\[ \coprod_{(w_1, \ldots, w_{\theta(1,\cu)})} \Hom(\Delta[m], 
	X[q](c_1, \ldots, c_q)(w_1, \ldots, w_{\theta(1,\cu)})). \]
	
	Now, if we want to take the colimit of the diagram \eqref{qmqc}, then we can merge these two descriptions to see that $\Hom(Q_{m,\cu}^{2,S}, X)$ is isomorphic to
	\[ \coprod_{(v_0, \ldots, v_q)} \coprod_{(w_1, \ldots, w_{\theta(q, \cu)})} \Hom(\Delta[m], X[q](c_1, \ldots, c_q)(v_0, \ldots, v_q)(w_1, \ldots, w_{\theta(1,\cu)})). \]
	
	We summarize these findings, and the analogous result for $P_{m, \cu}^{2,S}$, as well as generalizations for $T \subseteq S$, in the following lemma.  For notational simplicity, we write $\vu=(v_0, \ldots, v_q)$ and $\wu = (w_1, \ldots, w_{\theta(1,\cu)})$.
	
	\begin{lemma}
		When $n=2$, for a fixed $m \geq 0$, object $[q](c_1, \ldots, c_q)$ of $\Theta_2$, and functor $X \colon \Theta_2^{\op} \rightarrow \SSets$, there are natural isomorphisms
		\[ \Hom(P_{m,\cu}^{2,S}, X) \cong \coprod_{\vu} \coprod_{\wu} \Hom \left(\partial \Delta[m], X[q](c_1, \ldots, c_q)(\vu)(\wu) \right) \]
		and
		\[ \Hom(Q_{m,\cu}^{2,S}, X) \cong \coprod_{\vu} \coprod_{\wu} \Hom \left( \Delta[m], X[q](c_1, \ldots, c_q)(\vu)(\wu) \right). \]
		Discretizing only at $T=\{[0]\}$, we have
		\[ \Hom(P_{m,\cu}^{2,T}, X) \cong \coprod_{\vu} \Hom \left(\partial \Delta[m], X[q](c_1, \ldots, c_q)(\vu) \right) \]
		and
		\[ \Hom(Q_{m,\cu}^{2,T}, X) \cong \coprod_{\vu} \Hom \left( \Delta[m], X[q](c_1, \ldots, c_q)(\vu) \right). \]
	\end{lemma}
	
	Now let us consider the general case of $Q_{m,\cu}^{n,T}$ for general $n$.  To this end, let us denote the number of maps $[1]^{(i)} \rightarrow [q](c_1, \ldots, c_q)$ by $\theta(i,\cu)$.  We further denote an element of the set $(X[1]^{(i)}_0)^{\theta(i,\cu)}$ by $\vu^{(i)}=(v_1^{(i)}, \ldots, v_{\theta(i,\cu)}^{(i)})$.  The argument above generalizes to give a proof of the following lemma.  
	
	\begin{lemma} \label{hompqx}
		Let $X \colon \Thetanop \rightarrow \SSets$, and let $T=\{[1]^{(i)} \mid 0 \leq i \leq j \} \subseteq S$.  There are natural isomorphisms
		\[ \Hom(P_{m,\cu}^{n,T}, X) \cong \coprod_{\vu^{(0)}} \cdots \coprod_{\vu^{(j)}} \Hom \left(\partial \Delta[m], X[q](c_1, \ldots, c_q)(\vu^{(0)})\cdots (\vu^{(j)})\right), \]
		and 
		\[ \Hom(Q_{m,\cu}^{n,T}, X) \cong \coprod_{\vu^{(0)}} \cdots \coprod_{\vu^{(j)}} \Hom \left(\Delta[m], X[q](c_1, \ldots, c_q)(\vu^{(0)})\cdots (\vu^{(j)})\right). \]
	\end{lemma}
	
	Now, for the general case, for any $n$ and $T$, consider the set
	\[ I_f^{n,T} = \left\{P_{m,\cu}^{n,T} \rightarrow Q_{m,\cu}^{n,T} \mid m \geq 0, [q](c_1, \ldots, c_q) \in \ob(\Theta_n) \right\}. \]
	The following result is the main technical point we need to prove Proposition \ref{forward}.
	
	\begin{lemma} \label{thesis4.1}
		Let $T=\{[1]^{(i)} \mid 0 \leq i \leq j \} \subseteq S$, and suppose that $f \colon X \rightarrow Y$ is a map of $\Theta_n$-$T$-Segal precategories with the right lifting property with respect to the maps in $I_f^{n,T}$.  Then the map $X[0] \rightarrow Y[0]$ is surjective and each map 
		\[ X[q](c_1, \ldots, c_q)(\vu^{(0)}) \cdots (\vu^{(j)}) \rightarrow Y[q](c_1, \ldots, c_q)(f\vu^{(0)}) \cdots (f\vu^{(j)}) \]
		is an acyclic fibration of simplicial sets for any object $[q](c_1, \ldots, c_q)$ of $\Theta_n$ and every choice of $\vu^{(i)} \in (X[1]^{(i)}_0)^{\theta(i,\cu)}$ for each $0 \leq i \leq j$.
	\end{lemma}
	
	\begin{proof}
		The fact that $X[0] \rightarrow Y[0]$ is surjective follows from the fact that $f$ has the right lifting property with respect to the map $\varnothing \rightarrow \Theta[0]$.  Thus, we need to show that there is a lift in any diagram
		\[ \xymatrix{\partial \Delta[m] \ar[r] \ar[d] & X[q](c_1, \ldots, c_q)(\vu^{(0)}) \cdots (\vu^{(j)}) \ar[d] \\
			\Delta[m] \ar[r] \ar@{-->}[ur] & Y[q](c_1, \ldots, c_q)(f\vu^{(0)}) \cdots (f\vu^{(j)}). } \]
		By assumption, we know there exist lifts for diagrams
		\[ \xymatrix{P_{m,\cu}^{n,T} \ar[r] \ar[d] & X \ar[d] \\
			Q_{m,\cu}^{n,T} \ar[r] \ar@{-->}[ur] & Y.} \]
		Equivalently, in the diagram
		\[ \xymatrix{\Hom(Q_{m,\cu}^{n,T}, X) \ar[r] & P \ar[r] \ar[d] & \Hom(P_{m,\cu}^{n,T}, X) \ar[d] \\
			& \Hom(Q_{m,\cu}^{n,T} Y) \ar[r] & \Hom(P_{m,\cu}^{n,T}, Y),} \]
		where $P$ denotes the pullback of the right-hand square, the top left-hand map is surjective.  Using Lemma \ref{hompqx}, we can write $P$ as the pullback of the diagram
		\[ \xymatrix{\coprod_{\vu^{(0)}} \cdots \coprod_{\vu^{(j)}} \Hom \left(\Delta[m], Y[q](c_1, \ldots, c_q)(f\vu^{(0)}) \cdots (f\vu^{(j)})\right) \ar[d] \\
			\coprod_{\vu^{(0)}} \cdots \coprod_{\vu^{(j)}} \Hom \left(\partial \Delta[m], Y[q](c_1, \ldots, c_q)(f\vu^{(0)}) \cdots (f\vu^{(j)}) \right) \\
			\coprod_{\vu^{(0)}} \cdots \coprod_{\vu^{(j)}} \Hom \left(\partial \Delta[m], X[q](c_1, \ldots, c_q)(\vu^{(0)}) \cdots (\vu^{(j)}) \right). \ar[u] } \]
		On each component, i.e., fixing each $\vu^{(i)}$, the surjectivity of the map from
		\[ \coprod_{\vu^{(0)}} \cdots \coprod_{\vu^{(j)}} \Hom \left(\Delta[m], X[q](c_1, \ldots, c_q)(\vu^{(0)}) \cdots (\vu^{(j)})\right) \]
		to the appropriate component of $P$ produces exactly our desired lift.
	\end{proof}
	
	Now we use this result to shift our attention back to the maps $I^{n,T}$.
	
	\begin{lemma} \label{thesis9.1}
		Let $T=\{[1]^{(i)} \mid 0 \leq i \leq j \} \subseteq S$, and suppose that $f \colon X \rightarrow Y$ is a map of $\Theta_n$-$T$-Segal precategories with the right lifting property with respect to the maps in $I^{n,T}$.  Then $f_0 \colon X[0] \rightarrow Y[0]$ is surjective and the maps
		\[ X[q](c_1, \ldots, c_q)(\vu^{(0)}) \cdots (\vu^{(j)}) \rightarrow Y[q](c_1, \ldots, c_q)(f\vu^{(0)}) \cdots (f\vu^{(j)}) \] 
		are acyclic fibrations for any object $[q](c_1, \ldots, c_q)$ of $\Theta_n$ and every choice of $\vu^{(i)} \in (X[1]^{(i)}_0)^{\theta(i,\cu)}$ for each $0 \leq i \leq j$.
	\end{lemma}
	
	\begin{proof}
		If $f$ has the right lifting property with respect to the maps in $I^{n,T}$, then it has the right lifting property with respect to all cofibrations.  In particular, it has the right lifting property with respect to the maps in $I_f^{n,T}$, since it is not hard to check that these maps are all levelwise monomorphisms of simplicial sets.  Therefore the result follows from Lemma \ref{thesis4.1}.
	\end{proof}
	
	Finally, we can prove Proposition \ref{forward}.  Note that the proof strategy is substantially different from the one used for Segal categories in \cite[\S 5]{thesis}.
	
	\begin{proof}[Proof of Proposition \ref{forward}.]
		Suppose that $f \colon X \rightarrow Y$ is a map with the right lifting property with respect to the maps in $I^{n,T}$.  It follows that $f$ has the right lifting property with respect to all cofibrations, in particular those that are also weak equivalences.  Therefore $f$ is a fibration.  We need to show that it is a Dwyer-Kan equivalence.
		
		First consider the case where $X$ and $Y$ are $\Theta_n$-$T$-Segal categories.  Then applying Lemma \ref{thesis9.1} when $q=1$ gives the desired result, after observing that $X[1](c)(v_0, v_1)$ is the union of all the $X[1](c)(v_0, v_1)(\vu^{(1)}) \cdots (\vu^{(j)})$, and that fibrations are preserved under disjoint union.
		
		If $X$ and $Y$ are not $\Theta_n$-$T$-Segal categories, then a Dwyer-Kan equivalence between them is defined in terms of their localizations $LX$ and $LY$.  Therefore, we need to show that the required condition still holds after localizing.
		
		Let us recall how the localization is obtained.  If $X$ is not local, then we take iterated pushouts 
		\[ \xymatrix{\partial \Delta[n] \times \Theta[q](c_1, \ldots, c_q) \cup \Delta[m] \times G[q](c_1, \ldots, c_q) \ar[r] \ar[d] & X \ar[d] \\
			\Delta[m] \times \Theta[q](c_1, \ldots, c_q) \ar[r] & X'} \]
		as well as the analogous pushouts along the localizing maps $V[1](\mathcal S_{n-1})$ and $\mathcal T'_{n,j}$; recall that the latter were defined just before Definition \ref{dkequivdefn}.  We want to show that the induced square on mapping objects is still a pushout square; for simplicity, let us denote the left-hand map above, or any of its analogues for maps in $V[1](\mathcal S_{n-1})$ and $\mathcal T'_{n,j}$, as $A \rightarrow B$, so that we consider the maps on components
		\[ \xymatrix{A[1](c)(v_0, v_1) \ar[r] \ar[d] & X[1](c)(v_0, v_1) \ar[d] \\
			B[1](c)(v_0, v_1) \ar[r] & X'[1](c)(v_0, v_1).} \]
		We show that this diagram is a homotopy pushout square via an application of Mather's Cube Theorem \cite[Thm.\ 25]{mather} to the following diagram:
		\[ \xymatrix{A[1](c)(v_0,v_1) \ar[rr] \ar[dr] \ar[dd] && X[1](c)(v_0, v_1) \ar'[d][dd] \ar[dr] & \\
			& A[1](c) \ar[rr] \ar[dd] && X[1](c) \ar[dd] \\
			B[1](c)(v_0, v_1) \ar'[r][rr] \ar[dr] && X'[1](c)(v_0, v_1) \ar[dr] & \\
			& B[1](c) \ar[rr] && X'[1](c).} \]
		We know that the front square is a homotopy pushout, and we want to know that the back square is also; it suffices to show that the top, bottom, and sides of the cube are all homotopy pullback squares.  We verify this fact for the top square, namely
		\[ \xymatrix{A[1](c)(v_0, v_1) \ar[r] \ar[d] & X[1](c)(v_0, v_1) \ar[d] \\
			A[1](c) \ar[r] & X[1](c), } \]
		and leave the argument for the others as an exercise for the reader.  First, observe that the right-hand map is an inclusion of components and therefore a fibration, so it suffices to show that the square is an ordinary pullback.  Using the descriptions of $A[1](c)(v_0, v_1)$ and $X[1](c)(v_0, v_1)$ as pullbacks, one can check the necessary universal property for $A[1](c)(v_0, v_1)$.  
		
		Now, after applying the functorial localization functor to the map $X \rightarrow Y$, we obtain that the induced map on mapping objects
		\[ (LX)[1](c)(v_0, v_1) \rightarrow (LY)[1](c)(v_0, v_1) \]
		is necessarily a weak equivalence, completing the proof.
	\end{proof}
	
	Now we need to prove the converse, namely Proposition \ref{backward}.  To do so, we can generalize a construction used for the analogous proof when $n=1$. 
	
	\begin{proof}[Proof of Proposition \ref{backward}]
		Suppose $f \colon X \rightarrow Y$ is a fibration and a weak equivalence, and that $T=\{[1]^{(i)} \mid 0 \leq i \leq j\}$, for some $j <n$.  We need to show that $f$ has the right lifting property with respect to the maps in $I^{n,T}$.  First consider the case in which, for every $[1]^{(i)}$ in $T$, the induced map of discrete simplicial sets
		\[ f \colon X[1]^{(i)} \rightarrow Y[1]^{(i)} \]
		is an isomorphism.
		
		First, let us factor $f$ in the Reedy model structure $\SSets^{\Thetanop}$ as
		\[ X \hookrightarrow Y' \overset{\simeq}{\twoheadrightarrow} Y \]
		in such a way that $Y'$ still has the required components discrete, for example as in the proof of Proposition \ref{discreteloc}.  Since the map $Y' \rightarrow Y$ is a Reedy weak equivalence and therefore a Dwyer-Kan equivalence, we can conclude by the two-out-of-three property that $X \rightarrow Y'$ is also a Dwyer-Kan equivalence.
		
		Since $f$ is assumed to be a fibration, a lift exists in the diagram
		\[ \xymatrix{X \ar[r]^= \ar[d]_\simeq & X \ar[d]^f \\
			Y' \ar[r] \ar@{-->}[ur] & Y.} \]
		Therefore $f$ is a retract of $Y' \rightarrow Y$ and therefore a Reedy acyclic fibration.  Thus $f$ has the right lifting property with respect to monomorphisms, and in particular to the maps in $I^{n,T}$.  Thus, our result holds when $X$ and $Y$ have isomorphic discrete components.
		
		Now, we consider the general case, where for each $[1]^{(i)}$ in $T$ the maps of discrete spaces $X[1]^{(i)} \rightarrow Y[1]^{(i)}$ are surjective but not necessarily isomorphisms.  Define $\Phi Y$ to be the pullback of the diagram
		\[ \xymatrix{\Phi Y \ar[r] \ar[d] & Y \ar[d] \\
			\cosk_T(X) \ar[r] & \cosk_T(Y). } \]
		Observe that 
		\[ (\Phi Y)[1]^{(i)} \cong X[1]^{(i)} \]
		for every $0 \leq i <n$, and that for every $[q](c_1, \ldots, c_q) \in \ob(\Theta_n)$ and every $\vu^{(i)} \in (X[1]^{(i)}_0)^{\theta(i,\cu)}$ for each $0 \leq i \leq j$, the map
		\[ (\Phi Y)[q](c_1, \ldots, c_q)(\vu^{(0)}) \cdots (\vu^{(j)}) \rightarrow Y[q](c_1, \ldots, c_q)(\vu^{(0)}) \cdots (\vu^{(j)}) \]
		is a weak equivalence of simplicial sets. 
		
		We claim that $X \rightarrow \Phi Y$ is both a fibration and a weak equivalence.  It is not hard to show that it is a Dwyer-Kan equivalence, so let us show that $X \rightarrow \Phi Y$ is a fibration.  Let $A \rightarrow B$ be a generating acyclic cofibration, so we know that a lift exists in any diagram
		\[ \xymatrix{A \ar[d]_\simeq \ar[r] & X \ar[d]^f \\
			B \ar[r] \ar@{-->}[ur] & Y} \]
		since $X \rightarrow Y$ is assumed to be a fibration.  But we want to know that this lift is compatible with the factorization of $B \rightarrow Y$ as the composite $B \rightarrow \Phi Y \rightarrow Y$, so we want to know that the lift exists in the diagram
		\[ \xymatrix{A \ar[r] \ar[d]_\simeq & X \ar[d] \ar[dr]^f & \\
			B \ar[r] \ar@{-->}[ur] & \Phi Y \ar[r] & Y.} \]
		Since $X \rightarrow Y$ is assumed to be a fibration, we know that the indicated lift exists, but we want to know that it makes the top left-hand square commute as well, namely that the lift is compatible with the factorization of $B \rightarrow Y$ as the composite $B \rightarrow \Phi Y \rightarrow Y$. 
		
		We know that $\Phi Y$ agrees with $Y$ except possibly on the spaces corresponding to $[1]^{(i)}$.  Since $A \rightarrow B$ is a monomorphism and 
		\[ X[1]^{(i)} \cong (\Phi Y)[1]^{(i)} \]
		for all $0 \leq i <n$, a lift exists in any diagram
		\[ \xymatrix{A[1]^{(i)} \ar[r] \ar[d] & X[1]^{(i)} \ar[d] \\
			B[1]^{(i)} \ar[r] \ar@{-->}[ur] & (\Phi Y)[1]^{(i)}. } \]
		This lift, together with the lift $B \rightarrow X$ in the previous diagram, guarantees that the latter lift is compatible with the factorization of $B \rightarrow Y$ through $\Phi Y$.  It follows that $X \rightarrow \Phi Y$ is a fibration.
		
		Now, since $X \rightarrow \Phi Y$ is a fibration and a weak equivalence that is the identity on objects, we know from the first part of the proof that it has the right lifting property with respect to the maps in $I^{n,T}$.
		
		Finally, we want to show that the map $\Phi Y \rightarrow Y$ has the right lifting property with respect to the maps in $I^{n,T}$.  Thus, we want to show that a lift exists in any diagram of the form
		\[ \xymatrix{\left( \partial \Delta[m] \times \Theta[q](c_1, \ldots, c_q) \cup \Delta[m] \times \partial \Theta[q](c_1, \ldots, c_q) \right)_T \ar[r] \ar[d] & \Phi Y \ar[d] \\
			\left(\Delta[m] \times \Theta[q](c_1, \ldots, c_q) \right)_T \ar[r] \ar@{-->}[ur] & Y.} \]
		We work levelwise, evaluating at objects of $\Thetanop$.  
		
		If we evaluate at an object $[1]^{(i)}$ of $T$, then using the fact that we have discretized at such an object, one can check that the left-hand map is an isomorphism of discrete simplicial sets.  Hence, the desired lift exists.
		
		If we evaluate at any other object of $\Thetanop$, namely one at which we have not discretized, then it follows from the definition of $\Phi Y$ that the right-hand vertical map is an isomorphism of simplicial sets.  Thus, we again get the desired lift.
		
		Thus $\Phi Y \rightarrow Y$ has the right lifting property with respect to the maps in $I^{n,T}$; since we have established the same property for the map $X \rightarrow \Phi Y$, we can conclude that the composite $X \rightarrow Y$ has the right lifting property with respect to the maps in $I^{n,T}$, completing the proof.
	\end{proof}

\end{document}